\documentclass[12pt]{amsart}
\usepackage{amsmath,amssymb,latexsym}
\usepackage{fullpage,hyperref}
\renewcommand{\c}{\mathcal}
\newcommand{\xd}[1]{x_{#1}(\delta_{#1})}
\theoremstyle{plain}
\newtheorem{theorem}{Theorem}

\newtheorem{conjecture}{Conjecture}

\newtheorem{conj}{Conjecture}
\newcommand{\sht}{{\on{sht}}}
\newcommand{\lng}{{\on{lng}}}

\newtheorem{thm}{Theorem}
\newtheorem{cor}{Corollary}

\newtheorem{prop}{Proposition}
\newtheorem{lem}{Lemma}
\newtheorem{rmk}{Remark}

\newtheorem{defn}{Definition}

\newcommand{\bpm}{\begin{pmatrix}}
\newcommand{\epm}{\end{pmatrix}}
\newcommand{\bsm}{\begin{smallmatrix}}
\newcommand{\esm}{\end{smallmatrix}}
\newcommand{\bspm}{\left(\begin{smallmatrix}}
\newcommand{\espm}{\end{smallmatrix}\right)}
\newcommand{\bm}{\begin{matrix}}
\renewcommand{\em}{\end{matrix}}
\newcommand{\bbm}{\begin{bmatrix}}
\newcommand{\ebm}{\end{bmatrix}}

\newcommand{\shortOne}{w[2431542345654234576542314354287654231435426543765428765431]}
\newcommand{\longOne}{w[24315423456542314354276542314354265437654287654231435426543765428765431]}
\newcommand{\nuFour}{[345678243546576]}

\newcommand{\bs}{\backslash}

\newcommand{\C}{\mathbb{C}}
\newcommand{\G}{\mathbb{G}}
\newcommand{\A}{\mathbb{A}}

\newcommand{\Z}{\mathbb{Z}}

\newcommand{\Ind}{\operatorname{Ind}}

\newcommand{\St}{\operatorname{St}}

\newcommand{\st}{\text{ s.t. }}

\newcommand{\tx}{\text}
\newcommand{\la}{\langle}
\newcommand{\ra}{\rangle}

\newcommand{\ve}{\varepsilon}

\newcommand{\on}{\operatorname}
\newcommand{\ol}{\overline}

\renewcommand{\Re}{\on{Re}}

\newcommand{\gm}{\gamma}

\newcommand{\sg}{\sigma}
\newcommand{\quo}[1]{#1(F)\bs #1(\A)}

\newcommand{\il}{\int\limits}
\newcommand{\Supp}{\on{Supp}}

\newcommand{\f}{\mathfrak}
\renewcommand{\o}{\f o}
\newcommand{\ssm}{\smallsetminus}

\begin{document}
\date{}

\title{\bf A Doubling Integral for $G_2$}
\author{\bf  David Ginzburg and Joseph Hundley  }

\thanks{The first named author is partly supported by a grant from the Israel
Science Foundation no. 2010/10
The second named author was supported  by NSA
Grant MSPF-08Y-172
and NSF Grant DMS-1001792 during the period this work was completed.
}

\address{School of Mathematical Sciences\\
Sackler Faculty of Exact Sciences\\ Tel-Aviv University Israel
69978\\ and \\
Mathematics Department,  MC 4408\\
Southern Illinois University Carbondale\\
1245 Lincoln Dr. \\
Carbondale, IL 62901 }

\maketitle \baselineskip=18pt

\tableofcontents

\section{ Introduction}
In this paper we introduce a new global integral which represents
the standard $L$ function attached to a cuspidal representation of
the exceptional group $G_2(\A )$. Here $\A $ is the adele
ring of a global field $F$. In \cite{PS-R} the authors introduced a
global integral which represents the standard $L$ function for a
classical group  $H$. To describe their construction, for $i=1,2$,
let $\sigma_i$ denote two cuspidal representations of $H(\A )$.
Their global integral is given by
$$\int\limits_{H(F)\times H(F)\backslash H(\A )\times H({\A})}\varphi_{\sigma_1}(h_1)\varphi_{\sigma_2}(h_2)E((h_1,h_2),s)dh_1dh_2$$
Here $E(\cdot,s)$ is a certain Eisenstein series. Unfolding the
integral, one obtains the bilinear form
$$<\pi_1(h)\varphi_{\sigma_1},\varphi_{\sigma_2}> \ =\int\limits_{H(F)\backslash
H(\A )}\varphi_{\sigma_1}(h_1h)\varphi_{\sigma_2}(h_1)dh_1$$ as
inner integration. Using that, the authors established the fact that
the global integral is Eulerian. Moreover, assuming that $\sigma_1$
and $\sigma_2$  are contragredient, then this integral is nonzero
for some choice of data. Hence, the above global integral is nonzero
for {\sl all} cuspidal representation $\sigma_1$ of $H(\A )$.
Integrals of this type are now known as ``doubling
integrals.''

In this paper we construct a doubling integral which represents the
standard $L$ function for $G_2(\A )$. The global integral we
construct uses a certain Eisenstein series defined on the
exceptional group $E_8(\A )$. It is introduced in section \ref{section:  the global integral}
integral \eqref{global1}. Thus, as in the classical groups, we
obtain an integral construction which represents the standard $L$
function, and which is not zero for {\sl all} cuspidal
representations.
By attaching a character $\chi$ to the Eisenstein
series, our construction actually represents the twisted $L$
function.

A  global construction for this $L$ function is given in \cite{G1}, which is valid only for generic cuspidal representations. Another
construction for this $L$ function, which is valid also for cuspidal
representations which are not generic, is given in \cite{S}. In
contrast to our construction and to \cite{G1}, this global integral
unfolds to a non-unique model.

In this paper we prove the very basic properties of our
construction. In section \ref{section:  the global integral}, after some preparation we introduce the
global integral, unfold it, and show that it is Eulerian. In the
third section, using MacDonald's formula, we carry out the
unramified computations. In sections \ref{section:  normalizing factor} and \ref{s: calculation of I(s,t)} we carry out some
computations which we use in the proof of the unramified
calculations.
For some of the calculations, we used the software LiE \cite{L} and  another
software program \cite{egut}, written by the second named author.
This is mainly due to the sheer size
of the calculations, owing to the
large number of roots in $E_8,$
the size of the Weyl group, etc., as
well as to the large number of such calculations
which must be performed.
 It is worth noting that any individual calculation which was performed with software can also be
 checked by hand.

We also mention, that as in \cite{PS-R} pages 50-51,  the local
unramified computation gives us a definition of a generating
function for the standard $L$ function. Thus, it follows from
section \ref{section: the unramified computations} equation \eqref{unra1}, that if we define
$N(s)$ to be the normalizing factor of the Eisenstein series (cf. section \ref{section:  normalizing factor})
and we define
$$\Delta_s(g)=N(s)\int\limits_{U_0(F)}f(w_0zu(1,g),s\psi_U(u)du$$
then
$$\int\limits_{G_2(F)}\omega_\pi(g)\Delta_s(g)dg=L(\pi,s)$$

We expect that our global construction will have certain
applications in the study of this specific $L$ function, and its
twists by a character $\chi$. The first, is the study of the poles
of this $L$ function. The construction in \cite{G1} implies that for
generic cuspidal representations,  this $L$ function can have at
most a simple pole. As follows from section 4, one expects that the
Eisenstein series we use would have at most double poles. In fact
one expects that there are cuspidal representations, which are CAP
with respect to the Borel subgroup, whose $L$ functions will have a
double pole. Such CAP representations were constructed in
\cite{G-G-J}.

In the near future we hope to prove the following

\begin{conjecture}
The twisted partial standard $L$ function $L^S(\pi\otimes\chi,s)$
can have at most a double pole.
\end{conjecture}

\section{ The Global Integral}\label{section:  the global integral}

In this section we introduce the global integral, and carry out the
unfolding process.
We work with the unique split $F$-group of
type $E_8,$ which we assume to be equipped
with a choice of maximal torus $T$ and Borel subgroup
$B=TU_{\max}.$  Here $U_{\max}$ is a maximal unipotent subgroup
of $E_8.$
For $H \subset E_8$ a $T$-stable subgroup,
$\Phi(H,T)$ is the set of roots of $T$ in $H.$
Also for any reductive group $H$ with maximal
torus $S,$ $W(H,S)$ is the Weyl group of $H$
relative to $S.$  For $\Phi(E_8,T)$ and $W(E_8,T),$
we may write $\Phi$ and $W$ respectively.
In this paper we shall label the roots of $E_8$
by $\alpha_i, (1 \le i \le 8)$. The labeling we use is as in \cite{G-S}. Let $w_i$
denote the simple reflection corresponding to the root $\alpha_i$.
We shall denote the product $w_{i_1}\ldots w_{i_n}$ by $w[i_1\ldots
i_n]$.
Write $U_{\alpha}$ for the one-dimensional unipotent
subgroup attached to the root $\alpha,$ and equip $G$
with a realization $\{ x_\alpha: \G_a \to U_\alpha \mid \alpha \in \Phi(E_8, T)\},$ consisting of an isomorphism
$\G_a \to U_\alpha$ for each root $\alpha.$
We assume that
the structure constants are
determined as in \cite{G-S}.
For $1 \le i \le 8$, the product
$x_{\alpha_i}(1) x_{-\alpha_i}(-1) x_{\alpha_i}(1)$  is
a representative for the simple reflection $w_i.$  This,
in turn, determines a standard representative
in $G(F)$ for any word in the simple reflections.
We shall often abuse notation by conflating this
representative with the Weyl word it represents.

\subsection{Eisenstein Series}\label{ss:EisensteinSeries}
For $1 \le i \le 8$ we let $P_i$
denote the standard maximal parabolic subgroup
of $E_8$ such that $\alpha_i$ is a root of the unipotent
radical
and the remaining simple roots of $E_8$ are roots
of the standard
Levi factor.  We consider the group $P_2.$  It's Levi factor,
$M_2,$ is isomorphic to $\{ g \in GL_8: \det g \text{ is a square}\}.$  The group $M_2$ has a rational character
whose square is $\det.$
(Indeed, $M_2$ acts on the highest weight vectors in the
second fundamental representation of $E_8$
by such a representation.)
Denote this
rational character by $\det^{\frac 12}.$
 The modular quasicharacter $\delta_{P_2}$ of $P_2$ is
the  seventeenth power of $\det^{\frac 12}.$

Let
$\chi$ be a character of $\A^\times$ trivial
on $F^\times.$  Regard $\chi$ as a representation
of $M_2(\A)$ by composing with $\det^\frac 12.$
Consider the induced representation
$Ind_{P_2(\A)}^{E_8(\A )}\delta_{P_2}^s \chi.
$  This representation has an automorphic realization
as a space of Eisenstein series.  Specifically,
for $f_{s, \chi} \in   Ind_{P_2(\A)}^{E_8(\A )}\delta_{P_2}^s \chi$ the corresponding Eisenstein series
is defined by $E(h, f_{s, \chi}):=\sum_{\gm \in P_2(F) \bs E_8(F)} f_{s, \chi}( \gm g)$ for $\Re(s)$ large and by
meromorphic continuation elsewhere.

\subsection{Fourier Coefficient}\label{ss:FourierCoefficient}
Next we describe a Fourier coefficient which plays an important role
in this construction. As explained in \cite{G2} one can associate
with every unipotent orbit of a given group which defined over an
algebraic closed field, a set of Fourier coefficients. In \cite{G2}
it is explained how to do it in the classical groups, but it is
similar in the exceptional groups. The  Fourier coefficient we will
define is attached to the unipotent orbit $2A_2$ of $E_8$.

  In the notation
fixed above,  $P_1$ denotes the standard maximal parabolic subgroup of $E_8$
whose Levi factor  $M_1$
is the product of
a derived subgroup isomorphic to $Spin_{14}$
and a one-dimensional torus.
 Let $U$ denote its unipotent
radical of $P_1.$  Then $U$ is a two step unipotent group whose dimension is
78.  The center $Z(U)$ is $14$ dimensional and can be identified
with the standard (``vector'') representation of $Spin_{14}.$
The quotient $U/Z(U)$ is $64$ dimensional and can be
identified with a half-spin representation of $Spin_{14}.$

Let $\psi$ denote a nontrivial additive character of the group
$F\backslash \A $ .
The choice of $\psi$ identifies $F$ with
the Pontriagin dual of $F\bs \A.$
Characters of $U(\A)$  trivial on $U(F)$
are then identified with the
$F$-points of the
rational representation of $M_1$ which is dual to $U/Z(U).$
This would correspond to the other half-spin representation of $Spin_{14}.$

Over an algebraically closed field, this
representation of $M_1$ has
an open orbit.
As recorded in the tables on p. 405 of \cite{C}
and p.200 of  \cite{Kac}, the identity component
 is isomorphic to $G_2 \times G_2.$
 Moreover, using the discussion
 on
 p. 267 of \cite{Kimura}, it's not hard
 to show that the number of components
 in the stabilizer is two.
We now choose a character
of $U(F) \bs U(\A)$ which corresponds to a point in
general position.

Define the character $\psi_U$ of
$U(F)\backslash U(\A )$ as follows.
Using the fixed realization $\{x_\alpha\}$,
write $u \in U(\A)$ as
$u=x_{11221111}(r_1)x_{11122111}(r_2)x_{12232210}(r_3)x_{11233210}(r_4)u',$
where $u'$ is a product of elements $x_\alpha(u_\alpha)$
 corresponding to roots $\alpha$
 which are not among the
 four listed above. Then we define $\psi_U(u)=\psi(r_1+r_2+r_3+r_4)$.  It's not difficult to calculate
 the identity component of the stabilizer of this character, which turns out to
 be isomorphic to $G_2 \times G_2,$ proving that the
 character is indeed in general position.

To be precise, the images of
\begin{equation}\label{firstg2}
x_{\pm 00010000}(r);\ x_{\pm01000000}(r)x_{\pm00100000}(r)x_{\pm00001000}(-r);\
x_{\pm01010000}(r)x_{\pm00110000}(r)x_{\pm00011000}(r)
\end{equation}
$$x_{\pm01011000}(r)x_{\pm01110000}(r)x_{\pm00111000}(-r);\ x_{\pm01111000}(r);\
x_{\pm01121000}(r)$$
generate a subgroup of $E_8$ which is
isomorphic to $G_2.$  One may check that
an element of this group which
normalizes/centralizes its maximal torus also normalizes/centralizes the full maximal
torus of $E_8.$  This induces an embedding
of the Weyl group of this copy of $G_2$ into that
of $E_8.$
The images of the two simple reflections are $w_4$ and $w[235]$.

A second copy of $G_2$ is generated by
the images of the following homomorphisms:\begin{equation}\label{secondg2}
x_{\pm00000010}(r);\ x_{\pm00000001}(r)x_{\pm01011100}(-r)x_{\pm00111100}(r);\
x_{\pm00000011}(r)x_{\pm01011110}(r)x_{\pm001111100}(-r)
\end{equation}
$$x_{\pm01122210}(r)x_{\pm01011111}(r)x_{\pm00111111}(-r);\ x_{\pm01122211}(r);\
x_{\pm01122221}(r).$$
In this case, the simple reflections in the Weyl group
map to  $w_7$ and $w[865423456]$ in the
Weyl group of $E_8.$

Finally, suitable representatives for the Weyl element $w[657486576]$  stabilize $\psi_U$
and act on the two copies of $G_2$ by reversing them.

We shall denote the stabilizer of $\psi_U$ by $H^\pm$
and its identity component by $H.$
Note  that $T \cap H$
and $U_{\max} \cap H$ are
 a maximal torus and maximal unipotent
 subgroup of $H,$ respectively.
The one parameter subgroups listed above
equip each copy of $G_2$ with a realization
and the structure constants of the two realizations.  Thus, by equipping $G_2$ (regarded as an abstract group) with a realization
having the same structure constants, we pin down
a specific identification of $H$ with $G_2 \times G_2. $

One may now form the Fourier coefficient mapping
$$\varphi \mapsto \varphi^{(U, \psi_U)}(h):= \int\limits_{U(F)\backslash U(\A )}\varphi(uh)\psi_U(u)du,
\qquad \left( \varphi \in C^\infty(HU(F) \bs E_8(\A))\right)
$$
with image in the space
$C^\infty( H(F)U(\A)_{\psi_U} \bs E_8(\A)),$ of smooth,
left $H(F)$-invariant, $(U(\A), \psi_U)$-equivariant  functions $E_8(\A) \to \C.$

\subsection{Global Integral}
In order to define a global integral, we apply the Fourier
coefficient mapping defined in section \ref{ss:FourierCoefficient} to the Eisenstein series of section \ref{ss:EisensteinSeries}.
The result is a smooth function of uniformly
moderate growth $H(F) \bs H(\A) \to \C.$
It may therefore be integrated over $H(F) \bs H(\A)$ against a cusp form.
 Thus, let $\pi=\pi_1\otimes \pi_2$
denote an irreducible cuspidal representation of $H(\A )$.

The global integral we consider is
\begin{equation}\label{global1}
\int\limits_{G_2(F)\bs
G_2(\A )}\ \int\limits_{G_2(F)\bs
G_2(\A )}\varphi_{\pi_1}(g_1)\varphi_{\pi_2}(g_2) \int\limits_{ U(F)\backslash U(\A)}E(u(g_1,g_2),f_{s, \chi})\psi_U(u)\,du\,dg_1\,dg_2
\end{equation}
Here $G_2\times G_2$ is embedded in $E_8$ as the stabilizer of the
character $\psi_U$ as described above. Also, $\varphi_{\pi_1}$ and
$\varphi_{\pi_2}$ are vectors in the space of $\pi_1$ and $\pi_2$
respectively.

\subsection{Main Global Result}
The main result of this section is the following
\begin{theorem}\label{th1}
Let $w_\lng$ denote the shortest
element of $W(M_2, T) w_\ell W(M_1, T),$
where $w_\ell$ is the longest element of $W.$
Let $\nu_0 = w\nuFour,$ and let $w_0 = w_{\lng}\nu_0.$  Let $U_0 = U \cap w_0^{-1} \ol{U}_{\max} w_0= \prod_{\alpha \in \Phi(U,T): w_0 \alpha < 0}U_\alpha.$  Denote
 $$<\pi_1(g)\varphi_{\pi_1},\varphi_{\pi_2}> \ =\int\limits_{G_2(F)\backslash
G_2(\A )}\varphi_{\pi_1}(g_1g)\varphi_{\pi_2}(g_1)dg_1,$$
and let
$z=x_{00011100}(1)x_{00001110}(1)x_{00000111}(-1)$.
Then 
 the integral \eqref{global1} is equal
to
\begin{equation}\label{global2}
\int\limits_{G_2(\A )}\int\limits_{U_0(\A)}<\pi_1(g)\varphi_{\pi_1},\varphi_{\pi_2}>f_{s,\chi}(w_0zu(g,1))\psi_U(u)\,du\,dg.
\end{equation}
for all $\text{Re}(s)$ large.
\end{theorem}
\begin{proof}
Throughout this proof, we assume $s$ lies in the
domain of absolute convergence for $E(f_{s, \chi},g).$
For $\gm_0 \in E_8(F)$ define
$$E_{\gm_0}(g, f_{s, \chi}):=
\sum_{\gm \in (\gm_0^{-1} P_2 \gm_0 \cap UH)(F) \bs UH(F)} f_{s, \chi}(\gm_0 \gm g).
$$
Clearly, the sum is absolutely
convergent.  Moreover $E_{\gm_0}(g, f_{s, \chi})$ is left
$UH(F)$-invariant for each $\gm_0$ and
$$
E(g, f_{s, \chi}) = \sum_{ \gm_0 \in P_2(F)\bs E_8(F) / UH(F) } E_{\gm_0}(g, f_{s, \chi}).
$$
Since $E_{\gm_0}$ is $UH(F)$-invariant, we can consider
its Fourier coefficient $E_{\gm_0}^{(U, \psi_U)}.$

Rather than look for an exact set of representatives for the
double cosets $P_2(F)\bs E_8(F)/UH(F)$  it is convenient to work with a
larger
subset of $E_8(F)$ which clearly contains a set of representatives.
Write $L_{4,7}$ for the
subgroup of $E_8$  generated by $U_{\pm \alpha_4}, U_{\pm \alpha_7}.$  Clearly,
$L_{4,7} \subset H.$
Let $\ol{U}_{\max} = \prod_{\alpha < 0} U_\alpha$ denote the maximal unipotent
subgroup of $E_8$ opposite $U_{\max}$ and for $w \in W$
let $N_w= U_{\max} \cap w^{-1} \ol{U}_{\max}w = \prod_{\alpha > 0 , \; w\alpha < 0 } U_{\alpha}.$  Let $$\dot W(M_2, L_{4,7}) = \{ \sg \in W(E_8,T): \sigma \text{ is of minimal length in
}W(M_2, T) \cdot \nu \cdot \la w[4],w[7]\ra\sg\}.$$  Here $\la w[4],w[7]\ra$
denotes the subgroup of $W(E_8, T)$ generated by the $w[4]$ and $w[7].$
Then it follows from the Bruhat decomposition that the  set
\begin{equation}\label{set of reps sg delta}
\{ \sg \delta\mid,\sg \in \dot W(M_2, L_{4,7}) \; \delta \in M_1 \cap N_\sg(F)\}.
\end{equation}
is a set of double coset representatives for $P_2(F) \bs E_8(F)/ U(F)L_{4,7}(F),$
and hence surjects onto $P_2(F) \bs E_8(F)/U(F)H(F).$

We say that $w \in W$ is  left $M_2$ reduced if it is the shortest element of $W(M_2, T) \cdot w.$
Assume that $\gm_0 = w \mu,$ with
$w \in W$
 left ${M_2}$ reduced
and $\mu \in M_1(F).$  This certainly
includes the case $\gm_0=\sg \delta$ as in \eqref{set of reps sg delta}.
Then it can be shown that  $\gm_0 uh \gm_0^{-1} \in P_2$
if and only if $\gm_0 u\gm_0^{-1} \in P_2$ and $\gm_0 h\gm_0^{-1} \in P_2.$ It follows that
\begin{equation}\label{alternate expression for E_wmu^(U,psiU)}
\begin{aligned}
E_{\gm_0}^{(U, \psi_U)}(h, f_{s, \chi})&
= \il_{U^{\gm_0}(F)\bs U(\A)}
\sum_{\gm \in (H \cap \gm_0^{-1} P_2 \gm_0)(F) \bs H(F)}
f_{s,\chi}(\gm_0 \gm u h) \psi_U(u) \, du\\
&=\il_{U^w(F)\bs U(\A)}
\sum_{\gm \in (H \cap \gm_0^{-1} P_2 \gm_0)(F) \bs H(F)}
f_{s,\chi}(w u\mu \gm h) [\mu\cdot\psi_U](u) \, du.
\end{aligned}
\end{equation}
where, $
U^w := U \cap w^{-1} P_2 w,$
and
$
[\mu \cdot \psi_U](u) := \psi_U( \mu^{-1} u \mu)$ for $\mu \in M_1(F), \; u \in U(\A).
$

\begin{lem}\label{vanishing conditions}
Take $w \in W$ left $M_2$ reduced,
 and $\mu \in M_1(F).$
\begin{enumerate}
\item If the character $\mu\cdot \psi_U$ is nontrivial on the
group $U^w(\A),$ then $E_{w\mu}^{(U, \psi_U)}=0.$
\item  Set $\gm_0 = w \mu.$
If the group $(H \cap \gm_0^{-1} P_2\gm)$
contains the unipotent radical of a parabolic subgroup
of $H,$ then $E_{\gm_0}^{(U, \psi_U)}$ is orthogonal
to cuspforms $\quo H \to \C.$
\item For any $\gm_0 \in E_8(F),$ if the function $E_{\gm_0}^{(U, \psi_U)}$
is zero (resp. orthogonal to cuspforms), then $E_{\gm_0'}^{(U, \psi_U)}$
is zero (resp. orthogonal to cuspforms)
for all  $\gm_0' \in P_2(F)\cdot \gm_0 \cdot U(F)H^{\pm}(F).$
\end{enumerate}
\end{lem}
\begin{rmk}  ``Orthogonal to cuspforms'' just means that the integral
\begin{equation}\label{Jgm0}
J_{\gm_0}(f_{s, \chi}, \varphi):=
\int_{\quo H} E_{\gm_0}^{(U, \psi_U)}(h, f_{s, \chi})\varphi(h)\,dh,
\end{equation} is zero for any $f_{s, \chi} \in \Ind_{P_1(\A)}^{E_8(\A)}\delta_{P_2}^s \chi$ and any cuspform $\varphi: \quo H \to \C.$  Note that \eqref{Jgm0}
is precisely the contribution of $E_{\gm_0}$ to
our global integral.
The function $E_{\gm_0}^{(U, \psi_U)}(f_{s, \chi})$
need not be $L^2,$ so strictly speaking there is no
inner product space here.  However, it is of uniformly moderate growth, so it follows from the decay properties of cuspforms that \eqref{Jgm0} is absolutely convergent.
\end{rmk}
\begin{proof}\begin{enumerate}
\item
The function
$g \mapsto f_{s, \chi}(w g)$
is left-invariant by $U^w(\A).$
So, \eqref{alternate expression for E_wmu^(U,psiU)}
can be written as a double integral
with the inner integration being $$\int_{\quo{U^w}}[\mu\cdot \psi_U](u)\, du.$$
\item  The group
$H \cap \gm_0^{-1} P_2 \gm_0$
normalizes $U$ and also the subgroup $U\cap
\gm_0^{-1} P_2 \gm_0.$  Moreover,
 $H(\A)$ stabilizes $\psi_U.$ The function
 $f_{s, \chi} $ is left-invariant, by the $\A$-points of any unipotent subgroup of $P_2.$
 Consequently, the function
$$h \mapsto \il_{U^w(F)\bs U(\A)}
f_{s,\chi}(\gm_0 uh) \psi_U(u)\, du$$
is left-invariant by the $\A$-points of
any unipotent subgroup
of $H \cap \gm_0^{-1} P_2 \gm_0.$
The second part follows.
\item  It follows from the definition
of $E_{\gm_0}$ that $E_{p \gm_0 h}= E_{\gm_0}$
for $p \in P_2(F)$ and $h \in H(F).$  For
$h^- \in H^\pm (F) \ssm H(F),$ one has
$E_{\gm_0 h-}(g, f_{s, \chi}) = E_{\gm_0}( h^-g, f_{s, \chi})$
for all $g \in E_8(\A).$
\end{enumerate}
\end{proof}
\begin{lem}\label{Criterion1:goesUp(wnu, psiRoots) has to be trivial}
Let
\begin{equation}\label{SuppPhipsi}
\Supp_{\Phi}(\psi_U) = \{11221111,11122111,12232210,11233210\}.
\end{equation}
Take $\gm_0 = \sg \delta$ as in \eqref{set of reps sg delta}.
If $\sg \alpha >0$ for some $\alpha \in \Supp_{\Phi}(\psi_U),$  then $E_{\gm_0}(f_{s, \chi}, \varphi) =0.$
\end{lem}
\begin{proof}
 If
$\alpha \in
\Supp_{\Phi}(\psi)$
the $\psi_U$ is nontrivial on $U_\alpha.$ The condition $\sg \alpha >0$ implies that $U_\alpha \subset U^\sg,$
and
ensures that $\psi_U$ is nontrivial on $U^\sg.$
What must be shown is that  $[\delta \cdot \psi_U]$ remains nontrivial on $U_\alpha(\A),$ for all $\delta\in N_\nu (F).$   This follows from the fact that $N_\nu$ is contained in the standard
maximal unipotent and no two elements of $\Supp_\Phi(\psi_U)$ differ by a positive root.
\end{proof}

\begin{prop}
\label{6576->25}
The set $\dot W(M_2, L_{4,7})$ as in  \eqref{set of reps sg delta}
has $6576$ elements, of which all but $25$ map at least one of the four roots listed in lemma
\ref{Criterion1:goesUp(wnu, psiRoots) has to be trivial}
to a positive root.
\end{prop}
\begin{proof}
One can check this using the computer package LiE \cite{L}.
\end{proof}
Let $S= \{ \sg \in \dot W(M_2, L_{4,7}):  \sg \alpha > 0 \qquad \forall \alpha \in \Supp_\Phi(\psi)\}.$
According to proposition \ref{6576->25},
$S$ has $25$ elements.
Now, for $\sg \in \dot W(M_2, L_{4,7})$ and
$\delta \in (N_\sg \cap M)(F),$ one has
$P_2(F) \sg \delta L_{4,7}(F) \subset
P_2(F) \sg \delta UH(F) \subset P_2(F) \sg \delta P_1(F).$
So, it makes sense to sort these $25$ elements
according to the image in $P_2(F) \sg \delta P_1(F).$
Another straightforward computer check using
\cite{L} yields the next lemma.  \begin{lem}
If $\sg \in S$ then $P_2
\sg P_1$ contains either   \begin{equation}\label{shortOne}w_\sht:=\shortOne\end{equation} or  \begin{equation}\label{longOne}w_\lng:=\longOne.\end{equation}
If $S_{\star} = \{ \sg \in S: w_\star \in P_2 \sg P_1 \}, \; (\star = \on{sht},\on{lng}),$ then $S_\sht$ has $9$
elements and $S_\lng$ has $16.$  The $9$ elements of $S_\sht$
are all in the same $P_2(F), H(F)$ double coset,
and the  shortest of them is $w_\sht w[56].$
\end{lem}
\begin{prop}
For each $\sg \in S_{\sht},$ the restriction of $\delta \cdot \psi_U$ to $U^\sg$ is trivial if and only if $\delta \in H(F).$
\end{prop}
\begin{proof}
One can check this on a case-by-case basis using the
program ``DCA3'' from the egut package \cite{egut}.
\end{proof}
\begin{cor}
If $\sg \in S_\sht$ then $E_{\sg\delta}^{(U, \psi_U)}$
is orthogonal to cuspforms.
\end{cor}
\begin{proof}
We use all three parts of lemma \ref{vanishing conditions}.
It is clear that $E_{w_\sht w[56]}^{(U, \psi_U)}$ is
orthogonal to cuspforms by part (2).  Hence the same
holds for any element of $P_2(F) w_\sht H(F)$
by part (3), and the proposition shows that $E_{\sg\delta}^{(U, \psi_U)}$ is identically zero the rest of the
time by part (1).
\end{proof}
\begin{lem}
Suppose that $\sg \in S_\lng$ and $w_\lng^{-1} \sg$ can be written as a word in the simple reflections without using $w[4].$  Then $E_{\sg \delta}^{(U, \psi_U)}$ is orthogonal to cuspforms for all $\delta \in N_\sg \cap M_1(F).$  The same is true with ``$4$''
replaced by ``$7$''.
\end{lem}
\begin{proof}  Set $\nu = w_\lng^{-1} \sg.$
Since $w_\lng$ is left $M_1$-reduced, it follows that
$\sg \alpha < 0 \iff \nu \alpha <0$
for $\alpha \in \Phi(M_1, T).$
From this it follows that $N_\sg \cap M_1 = N_\nu.$

Let $P_4^{M_1} := {M_1} \cap P_4,$ i.e., the maximal standard parabolic subgroup of ${M_1}$ obtained by intersecting ${M_1}$ with the standard maximal parabolic subgroup $P_4$ of $G.$  Thus, the only simple root $\alpha$ of ${M_1}$ such that $U_\alpha$ is contained in the unipotent radical of $P_4^{M_1}$
is $\alpha_4.$
Suppose that $\nu$ can be expressed as a word in the simple reflections without using $4.$
Then $\nu,$ is actually in the Levi $M_4^{M_1}$ of $P_4^{M_1},$ and, consequently, so is  $N_\nu.$  Clearly, the unipotent
radical $U_4^{M_1}$ of $P_4^{M_1}$ is contained in $Q_w,$ and hence
in $\delta^{-1} \nu^{-1} Q_w \nu \delta$ for any $\delta \in N_\nu.$  Since the intersection of $U_4^{M_1}$ with $H$ is the unipotent radical of a parabolic subgroup of $H.$  The lemma follows from lemma \ref{vanishing conditions}, part (2).   The previous argument also works if ``$4$'' is replaced by ``$7$'' throughout, which completes the proof.
\end{proof}
Inspecting the $16$ elements of $S_{\lng}$ yields
a reduction.
\begin{lem}
Let $S_{\lng}'= \{ \sg \in S_\lng\mid w_\lng^{-1}\sg
\tx{ contains a $4$ and a $7$ }
\}$ then $S'_{\lng}$ has $8$ elements.
\end{lem}

\begin{prop}
Let $w_0 = w_\lng \nu_0$ where
$\nu_0 =w[345678243546576].$
 For three of the eight elements $\sg\in S_{\lng}',$ the function $E_{\sg \delta}^{(U, \psi_U)}$ is
orthogonal to cuspforms for all $\delta \in N_\sg \cap M_1.$  The remaining five have the property that
$P_2 \sg (N_\sg \cap M)$ intersects $P_2 w_0 N_{w_0} H^\pm,$ and $\delta \cdot \psi_U$
is nontrivial on $U^\sg$
unless $\sg \delta \in P_2(F) w_0 N_{w_0}(F) H^{\pm}(F).$
\end{prop}
\begin{proof}
This can be checked on a case by case basis using
the egut program DCA3.
\end{proof}

The results of the previous section imply that
$E_{\gm_0}^{(U, \psi_U)}$ is orthogonal to
cuspforms whenever $P_2(F)\gm_0 H^{\pm}
\cap w_0 (N_{w_0}\cap M_1)(F)$ is empty.  In this section
we must study $E_{\gm_0}^{(U, \psi_U)}$
for $\gm_0 \in w_0 N_{w_0}(F).$
Recall that $w_0 = w_\lng \nu_0$ where
$\nu_0 =w[345678243546576].$ And that
$N_{w_0}\cap M_1$ can be described more
simply as $N_{\nu_0}.$

We have
$$\Phi(N_{\nu_0}, T) =\left\{\begin{matrix}00000100,00000110,00000111,00001100,00001110,00001111,00011100,00011110,\\00011111,00111100,00111110,00111111,01122210,01122211,01122221\end{matrix}\right\}.$$
Write $\delta \in N_{\nu_0}$ as
$\prod_{\alpha \in \Phi(N_{\nu_0}, T)} x_\alpha( \delta_\alpha)$ with the product taken in the order the roots are
listed above.
Since $H$ contains every element of $E_8$ of the form
$$\begin{aligned}
x_{00000001}(r) x_{01011100}(-r) x_{00111100}(r),\qquad
x_{00000011}(r) x_{00111110}(-r) x_{01011110}(r),\\
x_{01011111}(r) x_{01122210}(r) x_{00111111}(-r),\qquad
x_{01122211}(r),\qquad\text{ or }\qquad
x_{01122221}(r),\end{aligned}
$$
It follows that every element of $\sg N_{\nu_0}(F)$ lies in the
same $P_2(F),H(F)$-double coset as $\nu_0 \delta$ for
some
$\delta$
such that
\begin{equation}\label{e:conditions on delta in nu4 case part 1}
\delta_{00111100}=\delta_{00111110}=\delta_{011222210}
=\delta_{011222211}=\delta_{011222221}=0.
\end{equation}
For such $\delta,$ the condition $\delta \cdot \psi_U \big|_{U^\sg} \equiv 1$ implies
     $$
\delta_{00011111}=
\delta_{00001111}=
\delta_{00000110}=
\delta_{00000100}=\delta_{00111111}=0$$
\begin{equation}\begin{aligned}
\delta_{00000111} &=
 \delta_{00001100} \delta_{00011110}- \delta_{00001110} \delta_{00011100}
\end{aligned}\label{e:Conditions on delta in the nu4 case}
     \end{equation}
\begin{lem}  The set
$$
\left\{
\xd{00001100}\xd{00011100}\xd{00001110}\xd{00000111}\xd{00011110}
\right\}
$$
 is a four-dimensional abelian subgroup $D_0\subset N_\nu,$ which is normalized
by
\begin{equation}\label{e: def of m0}
M_0  := (T\cap H)\cdot SL_2^{\alpha_4} \cdot SL_2^{\alpha_7}
\subset H\cap \nu^{-1} Q_w \nu.
\end{equation}
     The
     subset consisting of elements which satisfy \eqref{e:Conditions on delta in the nu4 case} is a union of three orbits for  this
     action,
     represented by  the identity, $x_{00011110}(1)$
     and \\$x_{00011100}(1)x_{00001110}(1) x_{00000111}(-1).$ \label{Lemma:  Action of M0 on D0}
    \end{lem}
\begin{proof}First, $\nu \cdot \alpha_4 =  \alpha_7,$
while $\nu \cdot \alpha_7 = \alpha_4.$  Both lie in $\Phi(L_w, T),$
and this proves that $M_0 \subset H\cap \nu^{-1} Q_w \nu.$  Verifying that the set $D_0$ is indeed a subgroup
is as simple as checking that no two of the roots
\begin{equation}\label{e:d0roots}00001100,00011100,00001110,00000111,00011110
\end{equation}
sum to give another root of $E_8.$
It is obvious that $T$ normalizes $D_0.$  Proving that
$SL_2^{\alpha_4} \cdot SL_2^{\alpha_7}$ does as well is as simple as verifying that when either $\alpha_4$ or $\alpha_7$ is subtracted from a root listed in \eqref{e:d0roots}, the result is either not a root, or one of the other roots listed in \eqref{e:d0roots}.

The action of $M_0$ on $D_0$ can be identified with the action of $GL_2^2$ on $\G_a\times \operatorname{Mat}_2$ via \begin{equation}\label{M0D0action}(g_1, g_2) \cdot (x, X) =
(\det g_1 \det g_2^{-1} x, g_1 X g_2^{-1}),\end{equation}
in such a way that the subvariety defined by
\eqref{e:Conditions on delta in the nu4 case} corresponds to that defined by $x = -\det X.$
It's clear that
\begin{itemize}
\item this subset is preserved by the action \eqref{M0D0action},
\item it is the union of three orbits, namely, the trivial orbit,
$\{(0,0)\}, \{(0,X): X\ne 0, \det X = 0\}$ and
$\{ (-\det X, X) : \det X \ne 0\}.$
\item the elements given above do indeed represent these orbits.
\end{itemize}
\end{proof}
\begin{cor}
If $\gm_0 \notin P_2(F) w_0 z H^\pm(F),$ then
 $E_{\gm_0}^{(U, \psi_U)}$
is orthogonal to cuspforms.
\end{cor}
\begin{proof}
Indeed, one can check using LiE that
$w_0 \alpha_i >0$ for $i=2,3,4,5.$  Hence
$w_0^{-1} H w_0$ contains the full unipotent radical of the copy of $G_2$ generated by \eqref{firstg2}.

Let $G_1$ denote the subgroup of $E_8,$
isomorphic to $Spin_8,$
generated by $U_{\pm \alpha_i}$ for $i = 2,3,4,5.$
Consider the maximal parabolic subgroup of
$G_1$  whose  unipotent radical contains $U_{\alpha_4}.$  Let $V_4$ denote this
unipotent radical.
Then
$x_{00011110}(1)V_4x_{00011110}(-1)\subset
V_4 U_{01121110}U_{01122110},$ and
$w_0 \alpha >0$ for each $\alpha \in \{01121110,01122110\}\cup \Phi(V_4, T).$  Hence
$w_0 x_{00011110}(1)$ conjugates
$(V_4 \cap H)$ into $P_2.$
The result follows from lemma \ref{vanishing conditions}, part (2).
\end{proof}
This completes the proof that
the global integral \eqref{global1} is equal to
$J_{\gm_0}(f_{s, \chi}, \varphi_{\pi_1} \cdot \varphi_{\pi_2}),$ defined by \eqref{Jgm0}, for
$\gm_0 = w_0 z.$
\begin{rmk}
For each $\gm_0 \in E_8(F),$   $J_{\gm_0}$
is a bilinear form between
$\Ind_{P_2(\A)}^{E_8(\A)} \delta_P^s \chi$ and the space of cuspforms on $\quo H.$
Define
$\Supp(J)$ to be the set of $\gm_0 \in E_8(F)$ such that $J_{\gm_0},$ is nonzero.  Then we have shown
that $\Supp(J)$ is a union of $P_2(F), UH^{\pm}(F)$ double cosets, and that it vanishes off of a single $P_2(F), UH$-double coset.  It is therefore worth asking whether $P_2(F) w_0 z UH^\pm(F) = P_2(F) w_0 z UH(F),$ for if this is not so, then we have proved that $J_{\gm_0}$ vanishes identically.  Thankfully, the answer is yes.  Indeed, if we define $^t\!z= x_{-00011100}(1) x_{-00001110}(1) x_{00000111}(-1),$ then $z^{-1}{}^t\!zz^{-1}$ is a representative for the Weyl element $w[657486576].$  Recall that certain representatives for this Weyl element lie in $H^\pm(F)\ssm H(F).$  Moreover one can choose $t_0 \in T(F)$ such that
$z^{-1}{}^t\!zz^{-1}t_0 \in  H^\pm(F)\ssm H(F)$ and
$t_0^{-1} z^{-1} t_0 =z.$  As $w_0 {}^t\!zt_0 w_0 \in P_2(F),$ it follows that $z\cdot ( z^{-1}{}^t\!zz^{-1}t_0 )$
and $z$ are in the same $P_2(F), H(F)$ double coset.
\end{rmk}

Let $J_{\gm_0}$ be defined
as in \eqref{Jgm0}.  Then
$$
J_{\gm_0}(f_{s, \chi}, \varphi) =
\il_{(H \cap \gm_0^{-1} P_2 \gm_0)(F) \bs H(\A)}
\varphi(h)
\il_{U^{\gm_0}(F) \bs U(\A)}
f_{s, \chi}( \gm_0 u h)\psi_U(u)
\, du
\, dh.
$$
We have shown that the global integral \eqref{global1}
is equal to $J_{w_0z}( f_{s, \chi}, \varphi_{\pi_1} \varphi_{\pi_2}).$
\begin{lem}
Define $w_0, z$ and $U_0$ as in theorem \ref{th1}.
Then $$\il_{U^{w_0z}(F) \bs U(\A)}
f_{s, \chi}( w_0z u h)\psi_U(u)
\, du
= \il_{U_0(\A)} f_{s, \chi}( \gm_0 u h)\psi_U(u)
\, du.
$$
\end{lem}
\begin{proof}
Note that
$U^{w_0} = \prod_{\alpha \in \Phi(U,T): w_0 \alpha >0} U_\alpha$ while $U_0 = \prod_{\alpha \in \Phi(U,T): w_0 \alpha <0}U_\alpha.$  It follows that $U = U^{w_0} U_0.$
One calculates
that $\{ \alpha \in \Phi(U,T): w_0 \alpha >0\}$ equals
$$\{ (11110000);\ (11111000);\ (11121000);\ (11221000);\ (12232100);\
(12232110);\ (12232111)\}.$$

For each value of $\alpha,$
one easily checks that none of $\alpha - 00011100, \alpha-00001110, \alpha-00000111$ is a root.
It follows that  $z$ normalizes $U_0,$
and hence $U = U^{w_0z} U_0.$

Since$f_{s, \chi}(w_0 zuh)$ is left-invariant by $U^{w_0z}(\A),$ and since $\psi_U$ is trivial on $U^{w_0z}(\A),$ the result follows.
\end{proof}
\begin{lem}\label{H cap w0 z P2 is diagonal G2}
The group $H \cap(w_0z)^{-1} P_2 w_0z$ is the
diagonally embedded copy of $G_2,$ i.e., the
subgroup of $H$ generated by
$
x_{\pm 00010000}(r) x_{\pm 00000010}(r),$ and
$$x_{\pm 01000000}(r) x_{\pm00100000}(r) x_{\pm00001000}(-r) x_{\pm 00000001}(r) x_{00111100}(r) x_{01011100}(-r).
$$
\end{lem}
\begin{proof}
The group $M_1 \cap w_{\lng}^{-1} P_2 w_\lng$
is the maximal parabolic subgroup of $M_1$
whose unipotent radical contains $U_{\alpha_3}.$
Denote this group by $P_3^{M_1}.$  Then
$H \cap(w_0z)^{-1} P_2 w_0z = H \cap (\nu_0z)^{-1} P_3^{M_1} \nu_0z= \{ h \in H \mid \nu_0z h z^{-1} \nu_0^{-1}\in P_3^{M_1}.$
Now, $M_1$ is isogenous to $SO_{14}.$
The kernel of this isogeny is contained in the maximal torus, and hence in every parabolic subgroup, and
this means $\nu_0z h z^{-1} \nu_0^{-1}$ lies in $P_3^{M_1}$ if and only if its image in $SO_{14}$ lies in
the corresponding parabolic subgroup of $SO_{14}.$
This reduces the lemma to a straightforward matrix calculation.
 \end{proof}
 This completes the proof of theorem \ref{th1}.
 \end{proof}
\section{ The Unramified Computations}
\label{section: the unramified computations}
\subsection{Notation}
In this section  the unramified
local zeta integral corresponding to integral \eqref{global2} will be computed.
Therefore,
let $F$ now denote a nonarchimedean local field,
and $\pi$ an unramified representation
of $G_2(F).$  In this section we write $T_{E_8}$
for our fixed maximal torus of $E_8$ and use $T$
for a fixed maximal torus of $G_2.$
Let $f$ denote a section of the induced representation
$\Ind_{P_2(F)}^{E_8(F)} \delta_P^s,$ where $s\in \C$
with $\Re(s)$ large.
(We no longer need to consider characters
of the type $\delta_P^s\chi$:  if $\chi$ is unramified,
then it is equal to $\delta_P^\nu$ for some $\nu \in
\C$ and we may absorb $\nu$ into $s.$ )
Let $\omega_\pi$ denote
the normalized spherical function for $\pi.$
Let $w_0=w_\lng \nu_0$ where
$$w_\lng=w[2431542345654231435427654231435426543765428765423143542654376542
8765431]$$ and $\nu_0=w[345678245673456]$.
Let $U$ denote the unipotent
radical of the standard maximal parabolic subgroup
of $E_8$ with Levi isomorphic to the
product of $Spin_{14}$ and a one dimensional torus,
and let   $U_0$ be
the
$T_{E_8}$-stable subgroup of
 $U$ such that
$
\Phi(U,T_{E_8}) \ssm \Phi(U_0, T_{E_8})$ equals $$
\{ (11110000);\ (11111000);\ (11121000);\ (11221000);\ (12232100);\
(12232110);\ (12232111)\}.$$
This group can also be described as $U \cap w_0^{-1} \ol{U}_{\max} w_0,$ or $ \prod_{\alpha > 0 : \, w_0 \alpha < 0 } U_\alpha.$
Define a character $\psi_U$ of $U$ (which may then be restricted to $U_0$)  by the formula
$$\psi_U(u)=\psi_U(x_{11221111}(r_1)x_{11122111}(r_2)x_{12232210}(r_3)x_{11233210}(r_4)u')=
\psi(r_1+r_2+r_3+r_4),$$
for any $u'\in U$ which lies in the product of the groups
$U_\alpha$ with $\alpha \in \Phi(U,T_{E_8})$ and $\alpha \notin \{11221111, 11122111, 12232210, 11233210\}.$
Embed $G_2\times G_2$ into $E_8$ as the
identity component of the stabilizer of $\psi_U$ in
$M_1,$ as in section \ref{ss:FourierCoefficient}.
Finally, let  $z=x_{00011100}(1)x_{00001110}(1)x_{00000111}(-1).$

Then the local zeta integral is given by
\begin{equation}\label{unra1}
I(s, \pi):=
\int\limits_{G_2(F)}\int\limits_{U_0(F)}\omega_\pi(g)f(w_0zu(1,g),s)\psi_U(u)\,du\,dg
\end{equation}
Throughout this section, we shall often abuse notation and denote the $F$-points of an algebraic group $H$
by ``$H$,'' suppressing the ubiquitous ``$(F)$'' 's.
We denote the long and short simple
roots of $G_2$ by $\beta_\lng$ and $\beta_\sht$
respectively.

We embed $G_2$ into $SO_8$ in such a way
that $$T=\{t=\text{diag}(t_1t_2^2,t_1t_2,t_2,1,1,t_2^{-1},t_1^{-1}t_2^{-1},t_1^{-1}t_2^{-2})\}.$$
This puts coordinates $t_1,t_2$ on $T.$  Another
way to define the same coordinates
is that $t_1 = \beta_\lng(t)$ and $t_2 = \beta_\sht(t).$

We also write $h:T \hookrightarrow E_8$
for the embedding
$h(\beta_\sht^\vee(a_1) \beta_\lng^\vee(a_2))
= \alpha_2^\vee\alpha_3^\vee \alpha_5^\vee(a_1)
\alpha_4^\vee(a_2)$
determined by our chosen embedding
$G_2 \hookrightarrow E_8.$  Composing
with the projection
$M_1\to SO_{14}$ gives the embedding
$$T \hookrightarrow SO_{14}, \qquad t\mapsto \begin{pmatrix} I_3&&\\ &t&\\ &&I_3\end{pmatrix},$$
which we also denote $h.$  (Here $T$ is identified
with a subgroup of $SO_8.$)

We also write $K$ for $G_2(\f o).$

\begin{rmk}
We saw in lemma \ref{H cap w0 z P2 is diagonal G2}
that the $H\cap (w_0z)^{-1} P_2
w_0 z$ is the diagonally embedded copy of $G_2$
inside of $H=G_2\times G_2.$  Hence $\{(1,g): g \in G_2\}\subset H$ maps isomorphically onto $(H\cap (w_0z)^{-1} P_2
w_0 z)\bs H.$  In fact we could take the integral
in \eqref{unra1} over any embedded copy of $G_2$
with this property.
\end{rmk}
\subsection{Main Local Result}
\label{ss: main local result}
\begin{thm}\label{main local theorem}
For $\Re(s)$ sufficiently large,
$$
I(s,\pi) = \frac{L(s, \pi, \St)}{N(s)},
$$
where $L(s, \pi, \St)$ is the local $L$ function
attached to $\pi$ and the $7$ dimensional
``standard'' representation of $G_2(\C),$
and  $N(s)$ is the normalizing factor of the
Eisenstein series, given explicitly (see section \ref{section:  normalizing factor}) by
$$N(s)=\zeta(17s)\prod_{i=2}^6\zeta(17s-i)\prod_{i=5}^8\zeta(34s-2i)\zeta(51s-21).
$$
\end{thm}
\begin{proof}
The proof occupies the rest of section \ref{section: the unramified computations}, and comes
in several steps.
Write $T$
for the maximal torus of $G_2,$
and write $T^+$ for the subset
$
\{ t \in T : |\beta_\lng(t)|, |\beta_\sht(t) | \le 1\}.$
As a subset of $SO_8,$ this is
$$\{\text{diag}(t_1t_2^2,t_1t_2,t_2,1,1,t_2^{-1},t_1^{-1}t_2^{-1},t_1^{-1}t_2^{-2})\in T: |t_1|, |t_2| \le 1\}.$$
For the rest of the section, we define
$t_1:= \beta_\lng(t)$ and $t_2:= \beta_\sht(t)$
as coordinates on the torus $T.$
Define $$\psi_{U,t}\left(x_{10111111}(r_1)x_{12232210}(r_2)x_{12232111}(r_3)
x_{11233210}(r_4)x_{11232211}(r_5)x_{11222221}(r_6)u'\right)=$$
$$\psi\left(\sum_{i=1}^4 r_i+t_2r_5+t_1t_2r_6\right),$$
$U_0':=\nu_0U_0\nu_0^{-1} = U \cap w_0^{-1} \ol{U}_{\max} w_0,$ and
\begin{equation}\label{equation:  def of I(s,t)}
I(s,t):=\int_{U_0'}f(w_0
u,s)\psi_{U,t}(u)\,du.
\end{equation}
Then we prove in section
\ref{ss: transformation of the integral}
that
\begin{equation}\label{I(s,pi) = int omega I(s,t) ...}
I(s, \pi) =\int_{T^+}\omega_\pi(t)\delta(t) |t_1t_2^2|^{17s-5}
I(s,t)
\,dt,
\end{equation}
where $\delta(t)$ is the measure
of the double coset $KtK.$

It's value is given
by
$\delta_B(t)^{-1} Q/Q_t,$
where
$$
Q = \frac{(1-q^{-2})(1-q^{-6})}{(1-q^{-1})^2}
= (1+2q^{-1}+2q^{-2}+2q^{-3}+2q^{-4}+2q^{-5}+q^{-6}),$$
$$Q_t :=
\begin{cases}Q, & |t_1|= |t_2|=1,\\
1+q^{-1}, & 0=\max(|t_1|, |t_2|)
>\min(|t_1|, |t_2|),\\
1& 0>\max(|t_1|, |t_2|).
\end{cases}
$$
This is proved in section
 3.2 of \cite{MacDonald}.  See also \cite{CasselmanNotesOnMacDonald}.

Every term in the integrand of \eqref{I(s,pi) = int omega I(s,t) ...} depends only on the $\f p$-adic
valuations of $t_1$
and $t_2$ of $T.$
To be precise, suppose that
$v(t_2)=m$ and $v(t_1)=n.$
Here $v$ is the $\f p$-adic valuation.
Then
$\delta_B(t) = q^{-6n-10m}.$
Define $Q_{(m,n)} = Q_t$ for $t \in T$
with $t_1= p^n$ and $t_2 = p^m.$

Now, set $x = q^{-17s},$ and define
$$Z(x,q) = (1-x)(1-xq^2)(1-xq^3)(1-xq^4)(1-x^2q^{10})(1-x^2 q^{12})$$ and
$$I_0(n,m; x,q) :=
(1-xq^6)(1-x^3q^{21})-(1-xq^5)(1-xq^6)(xq^8)^{m+1} - (1-xq^8)(1-xq^5)(xq^8)^m(xq^7)^{n+1}.
$$
Then theorem \ref{thm:  final formula for I(s,t)}
states that
$$I(s,t)=\frac{Z(x,q)I_0(n,m; x,q)}{(1-xq^7)(1-xq^8)}.$$
Moreover,
$|t_1t_2^2|^{17s-5}= (xq^5)^{2m+n}.$
This leaves $\omega_\pi(t).$  To write
a formula for $\omega_\pi(t),$ it is convenient
to identify the pair $(m,n)$ with the
weight $m\varpi_1+n\varpi_2$ in the weight lattice $\Lambda$   of $^LG_2^0 = G_2(\C)$.
This is compatible with
the natural identification
of $\Lambda$ with  $T/T(\f o)$ which is built into
the definition of the $L$-group.  Let $\tau$ be an element of the $L$-group associated
to the representation $\pi$ (this determines $\tau$ up
to the action of $W$ which is enough).
We write $\lambda$
additively, and therefore denote the
value of $\lambda \in \Lambda$ at $\tau$ by $\tau^{\lambda}$
as opposed to $\lambda(\tau).$
 Define $S_0$ to be the finite
set
of weights of $^LG^0$ which have at least one
expression as a sum of distinct positive roots,
and define polynomials $P_\nu,$ ($\nu \in S_0$)
by the condition
$$\left(\prod_{\alpha >0} 1-q^{-1} \tau^{-\alpha}\right)
=\sum_{\nu \in S_0}P_\nu(q^{-1})\tau^{-\nu}.$$
Then
$$
\omega_\pi(t) =
\frac{1}Q q^{-5m-3n}
\sum_{\nu \in S_0} P_\nu(q^{-1})
\frac{A_{\varpi+\rho-\nu}(\tau)}{A_\rho(\tau)},
\qquad \text{where} \quad
A_\lambda(\tau) = \sum_{w\in W}(-1)^{\ell(w)} \tau^{w\lambda}, \quad( \lambda \in \Lambda).
$$
Here $\ell$ denotes the length function on $W$
and $Q$ is defined as in the formula
for the measure of $KtK$ given above.
It is not difficult to derive
this expression for $\omega_\pi(t)$ from
the one given in \cite{Cartier}.

Plugging all of this into
\ref{I(s,pi) = int omega I(s,t) ...}
yields
$$
I(s,\pi)
=
\frac{Z(x,q)}{(1-xq^7)(1-xq^8)}
\sum_{m,n=0}^\infty\frac{x^{2m+n} q^{15m+8n}I_0(m,n; x,q)}{Q_{(m,n)}}
\sum_{\nu \in S_0} P_\nu(q^{-1})
\frac{A_{(m,n)-\nu+\rho}(\tau)}
{A_{\rho}(\tau)}.
$$
The inner summation must be calculated.  This
is accomplished in theorem \ref{thm:  sum p(n,m) I_0 = zetas * L }, which states that
\begin{equation}\label{sum with I0 = L fcn times zetas}
\sum_{m,n=0}^\infty\frac{x^{2m+n} q^{15m+8n}I_0(m,n; x,q)}{Q_{(m,n)}}
\sum_{\nu \in S_0} P_\nu(q^{-1})
\frac{A_{(m,n)-\nu+\rho}(\tau)}
{A_{\rho}(\tau)}
\end{equation}
$$
=(1-xq^5)(1-xq^6)(1-xq^7)(1-xq^8)(1-x^2q^{14})(1-x^3q^{21})\sum_{r=0}^\infty
\chi_{(r,0)}(\tau)x^rq^{8r},
$$
where $\chi_\lambda$ is the character
of the irreducible finite dimensional
representation of $G_2(\C)$ having highest
weight $\lambda.$
The decomposition of the symmetric
algebra of $\St$ is described by the
theorem of Brion \cite{Br}, and this
description implies that
$$L(s, \pi, \St) = (1-x^2 q^{16})^{-1} \sum_{r=0}^\infty
\chi_{(r,0)}(\tau)x^rq^{8r}.$$
Plugging this in, and noting that
$(1-xq^5)(1-xq^6)(1-x^2q^{14})(1-x^2q^{16})(1-x^3q^{21})Z(x,q)
$
is precisely $N(s)$ yields
 theorem \ref{main local theorem}.
\end{proof}

\subsection{Transformation of the integral}
\label{ss: transformation of the integral}
\begin{prop}
Let $I(s, \pi)$ be defined by \eqref{unra1},
and $I(s,t)$ by \eqref{equation:  def of I(s,t)}.
Then
$$
I(s, \pi) =\int_{T^+}\omega_\pi(t)\delta(t) |t_1t_2^2|^{17s-5}
I(s,t)
\,dt.
$$
\end{prop}
\begin{proof}
Using the bi-K invariant property of $\omega_\pi$, integral
\eqref{unra1} is equal to
\begin{equation}\label{unra3}
\int\limits_{T^+}\int\limits_{U_0}\omega_\pi(t)f(w_0zuh(t),s)\psi_U(u)\delta(t)\,du\,dt,
\end{equation}
where $\delta(t)$ denotes the measure of the double
coset $KtK.$

Conjugate $h(t)$ across $u.$  It normalizes
$U_0$ preserving $\psi_U,$ but
the change of variables $u \mapsto h(t) u h(t)^{-1}$
produces a factor of $$
\left|t^{\sum_{\alpha \in \Phi(U_0,T_{E_8})}\alpha }\right| = |t_1t_2^2|^{-5}.$$
from the measure.
(Here $t_1 = \beta_\lng(t)$ and $t_2 = \beta_\sht(t).$)

Next define
$z(t)= h(t)^{-1} z h(t).$  Then $z(t) = x_{00011100}(t_2)x_{00001110}(t_1t_2)x_{00000111}(t_1t_2^2)$.
This calculation
can be done in LiE or with matrices, since
the projection $M_1 \to SO_{14}$
restricts to an isomorphism on
$G_2 \cdot U.$
The image of $t$ was described above,
and the image of $z$ is
the
unipotent matrix  $I_{14}+e_{1,4}'+e_{2,5}'+e_{3,6}'$ in $SO_{14}$. Now
conjugate $z$ across $u_0$ to obtain that
integral \eqref{unra3} is equal to
\begin{equation}\label{unra31}
\int\limits_{T^+}\int\limits_{U_0}\omega_\pi(t)f(w_0uh(t)z(t),s)[z(t)\cdot \psi_U](u)\delta(t)dudt
\end{equation}
Notice that $z(t)$ is in the maximal compact and hence can be
ignored.

Write $w_0 = w_\lng\nu_0$ as in Theorem \ref{th1}.
We conjugate $h(t)$ to the left and then $\nu_0$ to the
right, and we obtain
\begin{equation}\label{unra32}
\int\limits_{T^+}\int\limits_{U_0'}\omega_\pi(t)f(w_\lng
u,s)[\nu_0 z(t)\cdot \psi_{U}](u)\delta(t)
|t_1t_2^2|^{-5} \delta_{P_2}^s(w_0 h(t) w_0^{-1})
\,du\,dt.
\end{equation}
 The character $\nu_0z(t) \cdot \psi_U(u)$ is precisely
$\psi_{U,t},$ and $\delta_{P_2}^s(w_0 h(t) w_0^{-1})
= |t_1t_2^2|^{17s}.$
This gives the result.
\end{proof}

We remark that  $U_0'$ may be defined as
$\nu_0 U_0 \nu_0^{-1},$  as $U \cap w_0^{-1} \ol{U}_{\max}$, or  as the
$T_{E_8}$-stable subgroup of $U$ such
that $\Phi(U,T_{E_8})\ssm \Phi(U_0', T_{E_8})$ is the
set
$$\{(12343210);\ (12343211);\ (12343221);\ (12343321);\ (12344321);\
(12354321);\ (13354321)\}$$ Thus $\text{dim}\ U=71.$

\subsection{Proof of Main Unramified Identity}
In this section we prove
\eqref{sum with I0 = L fcn times zetas}.
Set
$$I_0(n,m; x,q)=1-xq^6-x^3q^{21}+x^4q^{27}-x^{m+1}q^{8(m+1)}+x^{m+2}q^{8m+14}-x^{n+m+1}q^{7(n+1)+8m}$$
$$+x^{n+m+2}q^{7n+8m+12}+x^{n+m+2}q^{7n+8m+15}
-x^{n+m+3}q^{7n+8m+20}
+x^{m+2}q^{8m+13}-x^{m+3}q^{8m+19}
$$
$$
= (1-xq^6)(1-x^3q^{21}) - (xq^8)^{m+1}
(1-xq^5)(1-xq^6) -(xq^7)^{n+1}(xq^8)^m (1-xq^5)(1-xq^8)
$$
Identify $\Z^2$ with the weight lattice of
$G_2(\C)$ via the mapping $(n,m) \mapsto n \varpi_1 + m \varpi_2,$
and regard $I_0$ as a function
defined on the weight lattice as well.

Write $(n,m)$ for the trace of the semisimple
conjugacy class
in $G_2(\C)$
attached to $\pi,$ acting on the irreducible finite
dimensional
representation of $G_2(\C)$ with highest
weight $n\varpi_1+m\varpi_2.$  Here $\varpi_1$ and
$\varpi_2$ are the fundamental weights.

Define $S_0$ and $P_\nu$ as in the previous section
and set
$$Q_{\varpi} = \begin{cases}
Q, & \varpi = 0 \\
1+q^{-1}, & \varpi = n\varpi_1 \text{ or } n \varpi_2,\\
1, & \text{ otherwise,}
\end{cases}$$
for $\varpi$ a dominant weight of $G_2(\C).$

Define
$$p(\varpi) =
\frac{1}{Q_{\varpi}}
\sum_{\nu \in S_0} P_\nu(q^{-1})
\frac{\varpi+\rho-\nu}{|\varpi+\rho-\nu|} \chi_{|\varpi+\rho-\nu|}(\tau),
$$
and $p(n,m) = p( n\varpi_1+m\varpi_2).$

Define
$$z_0(x,q)=(1-xq^5)(1-xq^6)(1-xq^7)(1-xq^8)(1-x^2q^{14})(1-x^3q^{21}).$$

\begin{thm}\label{thm:  sum p(n,m) I_0 = zetas * L }
We have
\begin{equation}\label{check3}
\sum_{n,m=0}^\infty p(n,m)I_0(n,m; x,q)x^{n+2m}q^{8n+15m}=z_0(x,q)\sum_{r=0}^\infty
(r,0)x^rq^{8r}
\end{equation}\end{thm}
\begin{proof}
For $\varpi$ a dominant weight of $G_2(\C),$ define
$J(\varpi; x,q)
= I_0( \varpi; x,q) \tau_0^{\varpi}/Q_\varpi,$
where $\tau_0$ is an element
of the standard maximal torus $^LT$ of $G_2(\C),$ chosen
so that $\tau_0^{\varpi_1}= xq^8$ and $\tau_0^{\varpi_2} = x^2q^{15}.$   (Here, we employ the exponential notation
for the weights.  That is $\varpi$ is a function
from $^LT$ to $\C^\times,$ and its value at $\tau \in{} ^LT$
is denoted $\tau^\varpi.$)
Then
$$
J(\varpi; x,q)
= \frac1{Q_\varpi}
\left((1-xq^6)(1-x^3q^{21})\tau_0^\varpi
 -
(xq^8)(1-xq^5)(1-xq^6)\tau_1^\varpi -(xq^7) (1-xq^5)(1-xq^8)\tau_2^\varpi
\right),
$$
where $\tau_1, \tau_2 \in \;{}^LT$
satisfy
$$\tau_1^{\varpi_1}= xq^8, \quad
\tau_1^{\varpi_2} = x^3 q^{23},
\qquad
\tau_1^{\varpi_1}=x^2q^{15},
\quad \tau_2^{\varpi_2} = x^3 q^{23}.$$
Set
$$J_0(x, q) = (1-xq^6)(1-x^3q^{21}),
\qquad
J_1(x,q) = (xq^8)(1-xq^5)(1-xq^6),
\qquad
J_2(x,q) = (xq^7) (1-xq^5)(1-xq^8),$$
so that
$$J(\varpi; x,q) = J_0(x,q) \tau_0^{\varpi}
-J_1(x,q) \tau_1^{\varpi}
-J_2(x, q) \tau_2^{\varpi}.$$

Write $\Lambda^{++}$ for the semigroup of dominant
weights of $G_2(\C).$
For $\lambda \in \Lambda^{++},$ write $\chi(\lambda)$
for the character of the irreducible finite dimensional
representation of $G_2(\C)$ with highest weight $\lambda.$

Then $\chi(\lambda)$
appears in the expression for
 $p(\varpi)$ given in the
previous section if and only if
$\{(w,\nu)\mid w \in W, \nu \in S_0, \varpi + \nu + \rho = w(\lambda + \rho)\}$ is nonempty.
If this set is nonempty, then the coefficient
of $\chi(\lambda)$ in this expression
for $p(\varpi)$ is governed by
$$
D(\lambda, \varpi):=
\{ (w, S)\mid w \in W, \nu \in S_0, S \subset \Phi^+,
\varpi + \Sigma(S) + \rho = w(\lambda + \rho)\},$$
where $\Sigma(S):= \sum_{\alpha \in S}\alpha.$
To be precise
$$
p(\varpi) = \sum_{\lambda \in \Lambda^{++}}
\sum_{(w, S)\in D( \lambda, \varpi)}
(-1)^{\ell(w)} (-q^{-1})^{|S|}
$$
Thus,
it must be shown that for $\lambda\in \Lambda^{++}$
$$
\sum_{\varpi \in \Lambda^{++}}
\sum_{(w,S)\in D(\lambda, \varpi)}
J(\varpi; x,q) (-1)^{\ell(w)}
(-q^{-1})^{|S|}
=\begin{cases}
z_0(x,q) (xq^8)^r, & \varpi = r \varpi_1, \\
0, & \text{ otherwise.}
\end{cases}
$$
This will be reduced to a finite number of cases.
\begin{lem}
Given $\lambda, \varpi \in \Lambda^{++},$
$D(\lambda, \varpi)$ is nonempty
if and only if $\varpi \in \lambda + S_0.$
\end{lem}
\begin{proof}
The set $\rho-S_0 =\{ \frac12 \Sigma(S)-\frac12\Sigma(S^c) \mid S \subset \Phi\}$ is clearly $W$-stable.
(Here $S^c$ denotes the complement of $S$ in $\Phi.$)
So
$$\exists S \subset \Phi \st w(\varphi - \Sigma(S) + \rho)
= \lambda + \rho
\iff
\exists S \subset \Phi \st w\varpi -\Sigma(S)+\rho = \lambda+\rho
\iff w\varpi \in \lambda + S_0.
$$
Now, the convex hull of $\lambda+S_0$ is a convex dodecagon $D$
with edges parallel to the roots.  The center of
mass of the dodecagon $D$ is $\lambda+\rho$
and lies in the positive Weyl chamber.
Let $\ell$ be any line which is parallel to a root and does not pass through $\lambda+\rho.$  Then $\ell$
partitions the plane into two half-planes.  Let $H^+$
be the half-plane containing $\lambda+\rho$
and $H^-$ the other.  Then the reflection of $(D\cap H^+)$
over $\ell$ contains $H^-.$  Applying this geometric
observation to the reflections in the Weyl group,
it is clear that $\{ w \in W: w \varpi \in \lambda+S_0 \}$
must contain the identity if it is nonempty.
(In fact, one can say more.  If $w\varpi \in \lambda+S_0$
and $w=w[i_1i_2\dots i_k]$ is a reduced expression
for $w,$ then $w[i_j \dots i_k]\varpi \in \lambda+S_0$
for $1 \le j\le k.$)
\end{proof}
Thus we need to show that
\begin{equation}\label{sum J = cases}
\sum_{{\nu \in S_0}\atop{ \lambda + \nu \in \Lambda^{++}}}
J(\lambda + \nu ; x,q)
\sum_{{w \in W}\atop{w(\lambda+\rho) \in \lambda + \nu + \rho -S_0}}
(-1)^{\ell(w)}
\sum_{{S\subset \Phi}\atop{\lambda + \nu + \rho -\Sigma(S)=w(\lambda+\rho)}}(-q^{-1})^{|S|}=\begin{cases}
z_0(x,q) (xq^8)^r, & \lambda = r \varpi_1, \\
0, & \text{ otherwise.}
\end{cases}
\end{equation}
Now, $\{w\mid
w(\lambda+\rho) \in \varpi + \rho -S_0\}$
for some $\varpi \in \Lambda^{++}$
is given by
$$
\begin{cases}
\{e\}, & \lambda = n\varpi_1+m\varpi_2, \; n \ge 5, m \ge 3,\\
\{e,w[1]\}, & \lambda = n\varpi_1+m\varpi_2, \; n \le 4, m \ge 3,\\
\{e,w[2]\}, & \lambda = n\varpi_1+m\varpi_2, \; n \ge 5, m \le 2.
\end{cases}
$$
Here $e$ is the identity element of the Weyl group.
\begin{lem}
Fix $n$ with $0 \le n \le 4$ and $\nu \in S_0.$
For $\lambda = n\varpi_1+m\varpi_2,$
the expression
$$
P(n, \nu; q^{-1}):=
\sum_{{w \in W}\atop{w(\lambda+\rho) \in \lambda + \nu + \rho -S_0}}
(-1)^{\ell(w)}
\sum_{{S\subset \Phi}\atop{\lambda + \nu + \rho -\Sigma(S)=w(\lambda+\rho)}}(-q^{-1})^{|S|}
$$
is independent of $m \ge 3.$
Likewise, for fixed $m \le 2$ it is independent of $n \ge 5.$
\end{lem}
\begin{proof}
We prove the first statement.  The proof of the second
statement is symmetric.
Indeed, for $m \ge 3$ and $n \le 4,$ the
given expression is
$$
\sum_{{S\subset \Phi}\atop{ \nu =\Sigma(S)}}(-q^{-1})^{|S|}
+\sum_{{S\subset \Phi}\atop{\lambda + \nu + \rho -\Sigma(S)=w[1](\lambda+\rho)}}(-q^{-1})^{|S|},
$$
and $w[1]\lambda - \lambda = n \alpha_1,$
independent of $m.$
\end{proof}
Observe that for $n \le 4$ and $m \ge 3,$
$\{ \nu \in S_0\mid \lambda + \nu \in \Lambda^{++}\}$
is also independent of $m.$
Thus, for all $m \ge 3,$ the left hand side of equation
\eqref{sum J = cases} is equal to
$$\begin{aligned}
&\sum_{{\nu \in S_0}\atop{\nu = a\varpi_1+b\varpi_2, \, a \ge -n}}
J( \lambda+\nu; x,q) P(n, \nu; q^{-1})
\\
&=
\sum_{{\nu \in S_0}\atop{\nu = a\varpi_1+b\varpi_2, \, a \ge -n}}
(J_0(x,q)\tau_0^{\lambda+\nu}-J_1(x,q)\tau_1^{\lambda+\nu}-J_2(x,q)\tau_2^{\lambda+\nu}
)
P(n, \nu; q^{-1})\\
&=
(x^2q^{15})^m
\left(
(xq^8)^n
J_0(x,q)\sum_{\nu} P(n, \nu; q^{-1})
\tau_0^{\nu}
\right)\\&\qquad \qquad
-(x^3q^{23})^m
\left((xq^8)^n
J_1(x,q)\sum_\nu
P(n, \nu; q^{-1})\tau_1^{\nu}
+(x^2q^{15})^n
J_2(x,q)\sum_\nu
\tau_2^{\nu} P(n, \nu; q^{-1})\right).
\end{aligned}
$$
Now $m \mapsto (x^2q^{15})^m $
and $m \mapsto (x^3q^{23})^m$ are
two linearly independent functions of $m.$
Therefore, after checking equation \ref{sum J = cases}
for two distinct values of $m\ge 3,$ one may deduce
that
$$
\sum_{\nu} P(n, \nu; q^{-1})
\tau_0^{\nu}
=\left((xq^8)^n
J_1(x,q)\sum_\nu
P(n, \nu; q^{-1})\tau_1^{\nu}
+(x^2q^{15})^n
J_2(x,q)\sum_\nu
\tau_2^{\nu} P(n, \nu; q^{-1})\right)=0,
$$
and thence that \eqref{sum J = cases}
holds for all $m \ge 3.$
The same method allows one to
reduce the case when $m \le 2$ is fixed
and $n \ge 5$ is arbitrary to checking two
cases, and to reduce the case $m \ge 3, n \ge 5$
to checking three cases.  Overall, it suffices
to check all pairs $(n,m)$ with $n \le 6$ and $m \le 4.$
This is easily accomplished using LiE \cite{L}.
\end{proof}

\begin{rmk}
Equation \eqref{sum J = cases} has another simple proof
in the case $n \ge 5, m \ge 3.$ In that case, equation \eqref{sum J = cases} reduces
to
$$
\sum_{\nu \in S_0} J(\lambda+\nu; x,q)
\sum_{{S \subset \Phi^+}\atop {\Sigma(S)= \nu}}
(-q^{-1})^{|S|}=0.
$$
The left hand side is equal to
$$
\sum_{S\subset \Phi^+}
(-q^{-1})^{|S|} J(\lambda+\Sigma(S); x,q)$$
$$=J_0(x,q)
\prod_{\alpha \in \Phi^+}
(1-q^{-1} \tau_0^{\alpha}) \tau_0^\lambda-
J_1(x,q)
\prod_{\alpha \in \Phi^+}
(1-q^{-1} \tau_1^{\alpha}) \tau_1^\lambda-
J_2(x,q)
\prod_{\alpha \in \Phi^+}
(1-q^{-1} \tau_2^{\alpha}) \tau_2^\lambda.
$$
But one can fairly easily check that
for each $i=0,1,2,$ there is a positive root $\alpha$
such that $\tau_i^\alpha = q.$
\end{rmk}

\section{ The normalizing factor of the Eisenstein series}\label{section:  normalizing factor}

In this section we compute the normalizing factor
of the Eisenstein series.
The Eisenstein series  we use, denoted by $E(h,s)$, is attached to
the induced representation $Ind_{P_2(\A )}^{E_8(\A)}\delta_{P_2}^s\chi$.
Here, $\chi=\prod_v\chi_v$ is a character of $\A^\times/F^\times$
which has been identified with a character
of $P_2(F) \bs P_2(\A)$ by composing it with
the rational character $\det^{\frac12}$ of $P_2.$
(See section \ref{ss:EisensteinSeries}.)
To compute the normalizing factor we
consider an unramified place $v$, where $\chi_v = \delta_{P_2}^\nu$ for $\nu \in \C.$
Then $\Ind_{P(F_v)}^{E_8(F_v)}\delta_{P_2}^s\chi$
 is a sub-representation of
$Ind_{B(F_v)}^{E_8(F_v)}(\delta_{P_2}^{s+\nu}\delta_B^{-1/2})\delta_B^{1/2},$ where $B$
is the Borel of $E_8$. Any root $\alpha=\sum n_i\alpha_i$ such that
$n_2>0$ will contribute the factor
$$\frac{\zeta_v(17n_2(s+\nu)-\sum n_i)}{\zeta_v(17n_2(s+\nu)-\sum n_i+1)}
=\frac{L_v( 17n_2 s-\sum n_i, \chi_v^{n_2})}{
L_v( 17n_2 s-\sum n_i+1, \chi_v^{n_2})}.
$$
See \cite{PS-R}, proposition 5.2.
 Thus,
after cancellation we obtain the factor the product $Z_1(s)Z_2(s)$
where
$$Z_1(s)=\frac{L(17s-10,\chi)}{L(17s,\chi)}\frac{L(17s-11,\chi)}{L(17s-2,\chi)}
\frac{L(17s-12,\chi)}{L(17s-3,\chi)}\frac{L(17s-13,\chi)}{L(17s-4,\chi)}
\frac{L(17s-14,\chi)}{L(17s-5,\chi)}\frac{L(17s-16,\chi)}{L(17s-6,\chi)}$$
$$Z_2(s)=\frac{L(34s-17,\chi^2)}{L(34s-10,\chi^2)}\frac{L(34s-19,\chi^2)}{L(34s-12,\chi^2)}
\frac{L(34s-21,\chi^2)}{L(34s-14,\chi^2)}\frac{L(34s-23,\chi^2)}{L(34s-16,\chi^2)}
\frac{L(51s-29,\chi^3)}{L(51s-21,\chi^3)}$$ The denominator is the
normalizing factor, and for $\text{Re}(s)>\frac{1}{2}$ the poles of
the Eisenstein series should be determined by the poles of the
numerator. Thus we expect the following to hold.
\begin{conj}\label{poles}
For $Re(s)>1/2$, the
possible poles of the Eisenstein series
are as follows.
\begin{itemize}\item
If $\chi$ is trivial, then the
Eisenstein series $E(h,s)$  can have a double
pole at the points $\frac{10}{17},\ \frac{11}{17}$ and
$\frac{12}{17}$. At the points $\frac{9}{17};\ \frac{13}{17};\
\frac{14}{17};\ \frac{15}{17} $ and 1, it can have a simple pole.\item
If $\chi$ is nontrivial quadratic, then the Eisenstein series
can have simple poles at $\frac{10}{17},\ \frac{11}{17},\ \frac{12}{17},\ \frac{13}{17},\
\frac{14}{17},\ \frac{15}{17} $ and $1.$ \item
 If $\chi$
is nontrivial cubic, then the Eisenstein
series can have a simple pole at $\frac{10}{17}.$
\item If the order of $\chi$ exceeds $3$
then the Eisenstein series is holomorphic in
$\Re(s) >\frac 12.$
\end{itemize}
\end{conj}

\section{Calculation of $I(s,t)$}
\begin{thm}\label{thm:  final formula for I(s,t)}
Let
$I(s,t)$ be defined by \eqref{equation:  def of I(s,t)},
let $$Z(x,q) = (1-x)(1-xq^2)(1-xq^3)(1-xq^4)(1-x^2q^{10})(1-x^2 q^{12})$$ and
$$I_0(n,m; x,q) :=
(1-xq^6)(1-x^3q^{21})-(1-xq^5)(1-xq^6)(xq^8)^{m+1} - (1-xq^8)(1-xq^5)(xq^8)^m(xq^7)^{n+1}.
$$
Then
$$I(s,t)=\frac{Z(x,q)I_0(n,m; x,q)}{(1-xq^7)(1-xq^8)},$$
where $x=q^{-17s},\ n$ is the $\frak{p}$-adic valuation
of $t_1,$ and $m$ is that of $t_2.$
\end{thm}

\label{s: calculation of I(s,t)}\subsection{First reduction}
The purpose of this section is to compute
the period $I(s,t)$ appearing in \eqref{I(s,pi) = int omega I(s,t) ...}.
The first step is to  reduce the study of $I(s,t)$ to the study of
a simpler period $J(a,b,c),$  which we now define.

Throughout section \ref{s: calculation of I(s,t)},
we denote the maximal torus of $E_8$ by $T.$
If $w$ is an element of the Weyl group, $W,$
of $E_8,$ then
we define $U_w = U_{\max} \cap w^{-1} \ol{U}_{\max} w = \prod_{\alpha > 0: \; w\alpha < 0} U_\alpha.$
In this notation, the group $U_0'$ appearing in the
definition of $I(s,t)$ can also be described
as $U_{w_{\lng}}.$
\begin{defn}
For $a, b, c \in F^\times,$ define
\begin{equation}\label{J}
\begin{aligned}J(a,b,c) =\int_{U_w}f_s(wu)
\psi\left(u_{00100000} +a u_{01000000} +u_{10000000} +b u_{00001110} +c u_{00000111} \right)\,du,\\
w=w[243154234565423145765423187].
\end{aligned}\end{equation}
\end{defn}
\begin{thm}\label{theorem: I(s,t) in terms of J(a,b,c)}
$$
I(s, t)
=\begin{cases}
(1+q^{-51s+18}) J(1,1,1) ,& t_1,t_2\in \o^\times, \\
 J(1,1,t_1) - q^{-68s+26}
J(1,1,p^{-2}t_1) , & t_2 \in \o^\times, t_1 \notin \o^\times, \\
J(1, t_2, t_1t_2) - q^{-34s+14} J(p, p^{-1}t_2, p^{-1}t_1,t_2)
+ q^{-81s+35}J( 1, p^{-2} t_2, p^{-2} t_1t_2),
&t_2 \notin \o^\times.
\end{cases}
$$
\end{thm}

\subsubsection{Tools}\label{ss: tools}
Before proceeding
to the details of the proof we review  a few standard
techniques which are used in the calculations.
We consider integrals of the following type:
\begin{equation}\label{sectionIntegral}
\int_V f_s(wu) \psi_V(u) \, du,
\end{equation}
where $w\in W,$
$$V \subset \{ u = x_{\beta_1}(u_1) \dots, x_{\beta_N}(u_N):
u_i \in F,\; (1\le i \le N)\},$$
is defined by a finite set of
conditions of one of the following three types:
$$
|u_i|\le 1, \qquad |u_i|>1, \qquad u_i = c, \; c \in F, \text{ constant.}
$$
Also
$\beta_1, \dots, \beta_N$ are distinct roots
such that $w\beta_i \notin \Phi(P,T),\; (1\le i \le N),$
and
$$\psi_V(x_{\beta_1}(u_1) \dots, x_{\beta_N}(u_N))
= \psi\left( \sum_{i=1}^N c_iu_i\right),
\quad \text{ for some } c_1, \dots, c_N \in F.
$$
The variable $u_i$ is said to be {\it free} if it
does not appear in any of the conditions
which define $V.$
Of course, $I(s,t)$ is of this type.

The basic technique is to
 split an integral
of this type up according to
whether $|u_N|\le 1$ of $|u_N|>1$
and plug in the Iwasawa decomposition
for each to obtain two integrals of the
same type with a smaller value of $N.$
Some care must be taken with regard to
the order of the terms in the product
to avoid venturing outside of this relatively
simple class of integrals.

In addition to this basic technique,
there are a few additional tricks that can
be used.

Indeed, it's clear that the terms may be reordered
arbitrarily if we assume that
\begin{equation}\label{rearrangeable}\tag{R}
\beta_i+ \beta_j \text{ is a root }
\implies \begin{cases}
\beta_i+\beta_j = \beta_k \text{ for some }1\le k \le N,\\
c_k = 0,\\
u_k \text{ is free.}
\end{cases}
\end{equation}

\begin{lem}\label{obvious vanishing lemma}
Suppose that $u_N$ is not constant, and $c_N \notin \o.$  Then the section
integral \eqref{sectionIntegral} is zero.
\end{lem}
\begin{proof}
Introduce $x_{\beta_N}(z)$ at right with $z \in \o$ such
that $\psi(c_N z) \ne 1,$
and make a change of variables in $u_N.$
\end{proof}
\begin{lem}[Killing]\label{killing lemma}
Let
\begin{equation}\label{V'psiV'}
V' = \left\{x_{\beta_1}(u_1) \dots, x_{\beta_{N-1}}(u_{N-1})
 \right\}, \qquad
 \psi_{V'}(x_{\beta_1}(u_1) \dots x_{\beta_{N-1}}(u_{N-1}))
= \psi\left( \sum_{i=1}^{N-1} c_iu_i\right).
\end{equation}
Assume that
$$
\int_{V'} f(w v' x_{\beta_N}(u_N)x_\alpha(z))
\psi_{V'}(v') \, dv=
\psi( az+\ve u_N z)
 \int_{V'} f( wv' x_{\beta_N}(u_N))
\cdot \psi_{V'}(v')\, dv,
$$
where $a \in \f o $ and $\ve \in \f o^\times.$
Then one may restrict $u_N$ to $\f o$
without affecting the value of the integral.
In this situation we say we can {\bf kill} the
root $\beta_N.$
\end{lem}
\begin{proof}
Simply integrate $z$ over $\f o.$
\end{proof}
\begin{lem}\label{put p inverse}
Suppose that $c_N \in \o^\times,$ and that
there is a cocharacter $h_0: GL_1 \to T$ with
the property that
$\la h_0, \beta_N\ra = 1,$ and $\la h_0, \beta_i\ra = 0$
for all $i \ne N$ with $c_i \ne 0,$
and
 that the variable $u_N$ in \eqref{sectionIntegral}
is subject to the bound $|u_N|>1.$
Then the section integral
\eqref{sectionIntegral} is equal to  the integral
$$-\int\limits_{V'} f_s(wux_{\beta_N}(p^{-1})) \psi_{V'}(u) \, du.
$$
Here, $V'$ and $\psi_{V'}$ are defined as in \ref{killing lemma}.
\end{lem}
\begin{proof}
Introduce $h_0(\ve)$ at right and integrate
$\ve$ over $\f o ^\times.$  Conjugating $h_0(\ve)$
across and making suitable changes of variable,
one obtains an inner integration of
$$
q^k\int_{|\ve| = 1}\psi(c_N \ve p^{-k}) \, d\ve
= \begin{cases} -1, & k = 1\\
0 & \text{otherwise,}\end{cases}
$$
which gives the result.
\end{proof}
\begin{rmk}
A sufficient condition
for the existence of a cocharacter
with the given
properties is that
$\#\{i: c_i \ne 0\} \le 8,$ and that there is
an element of $SL(8,\Z)$ with the
property that each root $\beta_i$ with $c_i \ne 0$
is one of the rows.
 \end{rmk}

\begin{lem}\label{tying lemma}
Keep the notation and assumptions of lemma \ref{killing lemma}, but
now assume that
$$
\int_{V'} f(w v' x_{\beta_N}(u_N)x_\alpha(z))
\psi_{V'}(v') \, dv=
\psi( (a+bu_k+\ve u_N) z )
 \int_{V'} f( wv' x_{\beta_N}(u_N)
\cdot \psi_{V'}(v')\, dv,
$$
for some $a, b \in \o, \ve \in \o^\times,$ and $k \in \{ 1, \dots, N-1\}.$
Then the section integral \eqref{sectionIntegral}
is equal to
$$
\int_{V'}
f_s(w u)
\psi\left(\sum_{i=1}^{N-1} c_iu_i+c_N b u_k\right)
\, du,
$$
where $V'$ is the subset of $V$ defined by
the condition $u_N = - \ve^{-1}bu_k.$
\end{lem}
\begin{proof}
The proof is the same as that of lemma \ref{killing lemma}.
\end{proof}
\begin{lem}\label{shortening lemma}
Suppose that $w= w_1w_2.$
Then one may conjugate $w_2$ from
right to left without changing the value of
the integral.
\end{lem}
\begin{rmk}
In general $w_2 x_{\beta_i}(u_i) w_2^{-1}
= x_{w_2\beta}(\pm u_i).$  Normally, we make
changes of variables at the same time
to remove this signs.  This may introduce
signs into $\psi_V.$
\end{rmk}

If $S$ is any subset of $\Phi(G,T)$ which is closed
under addition, then the product of the groups
$U_{\alpha}, \; \alpha \in S,$ is a group, denoted $U_S.$
Note that $S=\Phi(U_S, T).$
  It will frequently be convenient
to describe a subgroup $U_S$ of $U_w$ by specifying the
complement of $\Phi(U_S, T)$  in $\Phi(U_w,T).$

\begin{proof}[Proof of first reduction theorem \ref{theorem: I(s,t) in terms of J(a,b,c)}]
Using lemma \ref{killing lemma}, one kills
\begin{equation}\label{first10killed}
   S=\left\{\bm  10100000
     ,10110000
     ,10111000
     ,10111100
     ,10111110,\\
     11110000
     ,11111000
     ,11111100
     ,11121000
     ,11221000\em\right\},
\end{equation}
deducing that
 $I(s, t),$
is equal to an integral of the same type,  a $61$-dimensional
subgroup $U'\subset U_0.$
let $S''$ be the
complement of $\{11111110,11121100,10111111,10000000\}$
in $S'.$
Number the elements of $S'',$ as
$\beta_1, \dots, \beta_{57},$ and
let
$$
U''= \{ x_{\beta_1}(u_{\beta_1}) \dots, x_{\beta_{57}}(u_{\beta_{57}}): u_i \in F, \; (1\le i\le 57)\}.
$$
Define
$$I\!I(s,t; m_1, m_2, m_3, n)=
\int\limits_{U''} f_s(w_\lng u x_{11111110}(m_2 )x_{11121100}(m_3 )x_{10111111}(m_1 )x_{10000000}(n ))
\psi_{U,t}(u) \, du.
$$
Thus
$$
I'(s,t) = \int\limits_F\int\limits_F\int\limits_F\int\limits_F
I\!I(s,t; m_1, m_2, m_3, n) \psi(m_1) \, dm_2\,dm_3\,dm_1\,dn.
$$
We then consider various cases, based
on the absolute values of $n, m_1, t_1,$ and $t_2.$
Write $I_{\o,\o}(s,t)$ for the integral over $n \in \o$
and $m_1\in \o,$ $I_{\o, F\ssm\o}(s,t)$ for the
integral over $n \in \o$ and $m_1\in F\ssm \o,$ and so
on, and $I_{F\ssm \o}$ for the integral
over $n \in F\ssm \o.$

Using lemma \ref{killing lemma} then lemma \ref{shortening lemma},
yields
$
I_{\o, \o} = J(1, t_2, t_1t_2).
$
\begin{prop}
One has
$$I_{F\ssm\o}(s,t) = \begin{cases}
q^{-51s+18} J(1,1,1), & t_1, t_2 \in \o^\times,\\
0, \text{ otherwise.}
\end{cases}$$
\end{prop}
\begin{proof}
The four terms in
$$x_{11111110}(m_2 )x_{11121100}(m_3 )x_{10111111}(m_1 )x_{10000000}(n )
$$
all commute with one another.  We rearrange
them as
$$x_{11111110}(m_2)x_{10000000}(n)x_{11121100}(m_3)x_{10111111}(m_1).
$$
And write
$I'_{F\ssm \o}(s,t)= I'_{F\ssm\o, \o}(s,t) + I'_{F\ssm \o, F\ssm o}(s,t),$ where
$$
\begin{aligned}
I'_{F\ssm\o, \o}(s,t)
&:= \int\limits_{F\ssm\o}\int\limits_\o\int\limits_F\int\limits_F
I\!I(s,t; m_1, m_2, m_3, n) \psi(m_1) \, dm_2\,dm_3\,dm_1\,dn\\
&= \int\limits_{F\ssm\o}\int\limits_F\int\limits_F
I\!I(s,t; 0, m_2, m_3, n) \, dm_2\,dm_3\,dn\\
I'_{F\ssm\o, F\ssm\o}(s,t)
&:= \int\limits_{F\ssm\o}\int\limits_{F\ssm\o}\int\limits_F\int\limits_F
I\!I(s,t; m_1, m_2, m_3, n) \psi(m_1) \, dm_2\,dm_3\,dm_1\,dn.
\end{aligned}
$$
In
$I'_{F\ssm\o, \o}(s,t),$ the root $01111110$
kills $11121100,$ and then, plugging in the Iwasawa
decomposition
of $x_{\alpha_1}(n)$ and simplifying
yields
$$\begin{aligned}
I'_{F\ssm\o, \o}(s,t)
&=
\int\limits_{F}
I\!I(s,t; 0, m_2, 0, 0)
\, dm_3\,
\int\limits_{F\ssm \o}
|n|^{-51s+17}
\, dn.\end{aligned}
$$
The first integral on the right hand side
 may also be obtained by killing one
 term in $I'_{\o,\o}.$  So it
 is equal to $I'_{\o,\o}$
 and hence to
 $J(1,t_2, t_1t_2).$

Next, lemma \ref{put p inverse} implies that
$$I_{F\ssm\o, F\ssm \o}(s,t)=
- \int\limits_{F\ssm\o}\int\limits_F\int\limits_F
I\!I(s,t; p^{-1}, m_2, m_3, n) \, dm_2\,dm_3\,dn.
$$
and then lemma \ref{tying lemma}
with $\alpha = 01111110$
yields
$$I_{F\ssm\o, F\ssm \o}(s,t)=
 \int\limits_{F\ssm\o}\int\limits_F
I\!I(s,t; p^{-1}, m_2, t_1t_2p^{-1}, n) \, dm_2\,dn.$$
Plugging in the Iwasawa decomposition for $x_{\alpha_1}(n)$ and simplifying gives
$$
\begin{aligned}
I_{F\ssm\o, F\ssm \o}(s,t)
&=-\int\limits_F
\int\limits_{U''} f_s(w_\lng u x_{11111110}(m_2 ))
\psi_{U,t}(u) \, du\, dm_2
\int\limits_{F\ssm\o}\psi\left(n^{-1} p^{-2} t_1 t_2^2\right)\,dn.
\end{aligned}
$$
As before, the integral on the left equals $J(1, t_2, t_1t_2).$
Now,
$$\int\limits_{|n|>1} |n|^{-51s+17} \psi(n^{-1} p^{-2} t_1 t_2^2 )
\,dn
= \int\limits_{|n|>1} |n|^{-51s+17}
\,dn,
$$
unless $t_1t_2^2 \in \o,$ in which case
$$\int\limits_{|n|>q} |n|^{-51s+17} \psi(n^{-1} p^{-2} t_1 t_2^2 )
\,dn
= \int\limits_{|n|>q} |n|^{-51s+17}
\,dn,
$$
while
$$\int\limits_{|n|=q} |n|^{-51s+17} \psi(n^{-1} p^{-2} t_1 t_2^2 )
\,dn
= -q^{-51s+17}.$$
Combined with the fact that
$$
\int\limits_{|n|=q} |n|^{-51s+17}
\,dn = q^{-51s+18}-q^{-51s+17},
$$
this yields the result.
\end{proof}

\begin{prop}
\begin{equation}
I_{\o, F\ssm \o}=
\begin{cases}
 0, & |t_1|= |t_2| =1, \\
 -J(p,p^{-1}t_2, p^{-1} t_1 t_2 )
q^{-34s+14}
+q^{-85s+35}J(1,p^{-2}t_2, p^{-2}t_1t_2), & |t_2|<1, \\
 -q^{-68s+26}J(1,1,p^{-2}t_1), & |t_2|=1, |t_1|< 1.
\end{cases}
\end{equation}
\end{prop}
\begin{proof}
This follows from the same type of arguments
using the lemmas from section \ref{ss: tools}.
We omit the details.
\end{proof}
Assembling the pieces we obtain the theorem.
\end{proof}

\subsection{Second Reduction}
\label{ss: J to J'}
\begin{prop}\label{prop J in terms of J'}
Let $f$ be the normalized spherical vector
in the induced representation attached to the
character
$$
\prod_{i=1}^8 \alpha_i^\vee(t_i)
\mapsto
\prod_{i=1}^8 |t_i|^{s_i},\quad\text{ with }$$
$$[s_1, \dots, s_8] =
[17s-6,17s-6,17s-6,-34s+14,17s-6,-17s+7,17s-6,17s-5],$$
and let
$$J_0(b,c) :=
\int_{U_{w[57687]}}
f(w[57687]u)
\psi\left(b u_{00001110} +c u_{00000111} \right)
\,du.
$$
Then
$$
J(a,b,c) =
\frac{\zeta(17s-6)\zeta(17s-7)}{
\zeta(17s-4)\zeta(17s-3)\zeta(17s-2)\zeta(17s)\zeta(34s-10)}
(1-|a|^{17s-7}q^{-17s+7})J_0(b,c).
$$
\end{prop}
\begin{proof}
First,
write
$w[243154234565423145765423187]=
w''w'$ with
$w''=w[243154234654237654], w'=w[131257687].$
We have
\begin{equation}
\label{J as integral with M(s,w'')}
J(a,b,c)
= \int_{U_{w'}}
(M(s, w'')f_s)(w'u)\psi\left(u_{00100000} +a u_{01000000} +u_{10000000} +b u_{00001110} +c u_{00000111} \right)\,du,
\end{equation}
where
$$
M(s,w'')f_s(g)= \int_{U_{w''}} f_s(w'' u g) \, du
$$
is the standard intertwining operator.  By the Gindikin-Karpelevich formula 
we have
$$
M(s,w'')f_s= \frac{\zeta(17s-6)^4\zeta(34s-13)}{
\zeta(17s-4)\zeta(17s-3)\zeta(17s-2)\zeta(17s)\zeta(34s-10)}f.
$$
Next, observe that $w'$ and $U_{w'}$ are contained
in the standard Levi subgroup of $E_8$ which has
derived group isomorphic to $SL_2 \times SL_3
\times SL_5.$
One may factor $w'$ as $w' = w[131]w[2]w[57687],$
and also factor the integral in \eqref{J as integral with M(s,w'')} into a product of $3$ simpler integrals
corresponding to the three components of the Levi and the three factors of $w'.$  The integrals corresponding
to $w[131]$ and $w[2]$ are Jacquet integrals, equalling
$\zeta(17s-6)^{-2}\zeta(34s-13)^{-1}$
and $\frac{\zeta(17s-7)}{\zeta(17s-6)}(1-|a|^{17s-7}q^{-17s+7}),$ respectively.
The proposition follows.\end{proof}

\subsection{An $SL_5$ period}
\label{ss: Calc of J'}
From Theorem \ref{theorem: I(s,t) in terms of J(a,b,c)}
and Proposition \ref{prop J in terms of J'}.
 the computation
 of $I(s,t),$ is reduced to the computation
 of an integral $J_0(b,c)$ over a unipotent subgroup of
 the copy of
$SL_5$ in $E_8$ which is generated by $x_{\pm\alpha_5};\ x_{\pm\alpha_6};\
x_{\pm\alpha_7}$ and $x_{\pm\alpha_8}$.
In this section we compute
$J_0(b,c).$

\begin{prop}
The function $J_0(b,c)$ depends only
on the $\frak{p}$-adic valuations of $b$ and $c.$
Take  integers $B$ and $C$ with $B \le C,$ and take $E$ equal to either an integer or $\infty.$
Define  $J_2(B,C,E)$  to be
$$\begin{aligned}
\frac{(1-xq^6)(1-(xq^7)^{B+1})}{(1-xq^7)^2(1-x^2q^{13})}
\begin{cases}
(1-x^2q^{12})(1-xq^6)-(1-xq^5)(1-x^2q^{13})(xq^7)^{C+1}&\min(B,C,E) \ge 0\\
\qquad\qquad +(1-q^{-1})(xq^6)(1-xq^7)
(x^2q^{13})^{C+1}, &\qquad  E \ge C, \\
(1-x^2q^{12})(1-xq^6)(1-(x^2q^{13})^{E+1})&\min(B,C,E) \ge 0,\\\qquad
-
(xq^7)^{C+1}(1-xq^5)(1-x^2q^{13})(1-(xq^6)^{E+1}),& \qquad E < C,\\
0& \min(B,C,E) < 0.
\end{cases}\end{aligned}
$$
where $x=q^{-s}.$
Then
for all $b,c \in F$ with $|b|=q^{-B}$ and $|c|=q^{-C},$
\begin{equation}\label{summations}
\begin{aligned}J_0(b,c)=&
 J_2(B,C,\infty )+ (1-q^{-1})\sum_{\ell=1}^B
 J_2(B-\ell, C-\ell, \infty) (xq^8)^{\ell}
 +(1-q^{-1})\sum_{k=1}^B(x^2q^{13})^k J_2(B-k,C,\infty)
 \\
&
 + (1-q^{-1})^2\sum_{k=1}^B
 \left[\sum_{\ell=0}^{k-1}
 J_2(B-k,C,C-k+\ell)q^{-\ell}+
\sum_{\ell=1}^{B-k}
 J_2(B-k-\ell, C-\ell, C-k-\ell) (xq^8)^{\ell} \right].
\end{aligned}\end{equation}
\end{prop}
\begin{proof}
Define
$$
n^-(x_1,x_2,x_3,x_4,x_5):=
x_{\alpha_5+\alpha_6+\alpha_7}(x_5)
x_{\alpha_6+\alpha_7+\alpha_8}(x_4)
x_{\alpha_7+\alpha_8}(x_3)
x_{\alpha_6+\alpha_7}(x_2)
x_{\alpha_7}(x_1)\in E_8.
$$
This gives an explicit parametrization
of the group $U_{w[57687]}.$
Next, define
$$
\hat J_1(b,c,e)=
\int_{F^4}
f_\chi(n^-(0,x_2,x_3,x_4,x_5)
\psi(bx_5+cx_4+ex_2x_3) \, dx
$$
$$
\hat J_2(b,c,e)=
\int_{F^4}
f_\chi(n^-(0,0,x_3,x_4,x_5))
\psi(bx_5+cx_4+ex_3) \, dx
$$
$$
\hat J_4(c,e) =\int_{F^2}f_\chi(n^-(0,0,x_3,x_4,0))
\psi(cx_4+ex_3) \, dx.
$$
Then by plugging in the Iwasawa decompositions
of $x_{\alpha_7}(x_1)$ and then $x_{\alpha_6+\alpha_7}(x_2),$ one finds that
\begin{equation}\label{J0->J1->J2}\begin{aligned}
J_0(b,c)&= \hat J_1(b,c,0)
+ \int_{F\ssm \f o}
\hat J_1(bx_1,b,bx_1)
|x_1|^{-34s+13}
\, dx_1
\\
\hat J_1(b,c,e) &=
\int_{\f o} \hat J_2(b,c,x_2)
\, dx_2 + \int_{F\ssm \f o}
\hat J_2( bx_2, cx_2, ex_2)|x_2|^{-17s+7}
\, dx_2 .
\end{aligned}\end{equation}
Moreover, if $1_{\f o}$ is the characteristic
function of $\f o,$ then
\begin{equation}\label{J4 to J2}
\hat J_2
(b,c,e)=\frac{\zeta(17s-7)}{\zeta(17s-6)}(1-|b|^{17s-7}q^{-17s+7})\hat J_4(c,e) \cdot 1_{\f o}(b).
\end{equation}
(The integration in $x_5$
 amounts to an $SL_2$ Jacquet integral.)
 Likewise $$\begin{aligned}\hat J_4(c,e)
 =
   \frac{\zeta(17s-7)}{\zeta(17s-6)}&\left(1_{\f o}(c)
 \int\limits_{\f o}
(1-|c|^{17s-7}q^{-17s+7})\psi(ex_3)\, dx_3 \right.\\
&\qquad\left.
+ \int\limits_{F\ssm \f o}1_{\f o}(cx_3)
(1-|cx_3|^{17s-7}q^{-17s+7})\psi(ex_3)|x_3|^{-34s+12}
\, dx_3
\right).\end{aligned}
 $$
 At this point it is clear that $\hat J_4,$ and hence all of the other integrals, depend only on the absolute values,
 or, equivalently, $\f p$-adic valuations, of their arguments.
Introducing the notation $x:=q^{-17s},$ we have$$
 \hat J_4(p^C,p^E) = 1_{\f o}(c)1_{\f o}(e)
  \frac{1-xq^6}{1-xq^7}\left(
(1-(xq^7)^{C+1})
+ \sum_{m=1}^{C}
(1-(xq^7)^{C-m+1})(x^2q^{12})^m
\int_{|x_3|=q^m}
\psi(ex_3)
\, dx_3\right).
 $$
 Denote this quantity by $J_4(C,E).$
Recall that for any $e \in F$ with $|e|=q^{-E},$
 $$\int_{|x_3|=q^m}
\psi(ex_3)
\, dx_3
=\begin{cases}
(1-q^{-1}) q^m, & m \le E, \\
-q^E, & m = E+1, \\
0, & m > E+1.
\end{cases}$$
It follows that $J_4(C,E)=0$ if either $C$ or $E$ is
negative, and that otherwise
$$\begin{aligned}
&J_4(C,E)
=
\frac{1}{(1-xq^7)(1-x^2q^{13})}
\begin{cases}
(1-x^2q^{12})(1-xq^6)-(1-xq^5)(1-x^2q^{13})(xq^7)^{C+1}&\\
\qquad\qquad +(1-q^{-1})(xq^6)(1-xq^7)
(x^2q^{13})^{C+1}, & E \ge C, \\
(1-x^2q^{12})(1-xq^6)(1-(x^2q^{13})^{E+1})&\\\qquad
-
(xq^7)^{C+1}(1-xq^5)(1-x^2q^{13})(1-(xq^6)^{E+1}),& E < C.
\end{cases}\end{aligned}
$$
It follows from \eqref{J4 to J2} that
  $$
 \hat J_2(b,c,e) = \frac{1-xq^6}{1-xq^7}
(1-(xq^7)^{B+1})J_4(C,E)1_{\f o}(b)
 =J_2(B,C,E),
 $$
 whenever $b,c\in F^\times, e \in F$ have valuations
 $B,C,E$ respectively (with the convention that
 the valuation of $0$ is $\infty$).
 Plugging in to \eqref{J0->J1->J2}, and
 using the fact that $B \le C$ and that the volume of $\{ y \in F: |y|=q^k \}$ is $q^k(1-q^{-1})$
for each $k\in \Z,$ gives the result.
 \end{proof}

Now,  take $X = (X_1, \dots, X_6)$ to be a sextuple of
indeterminates,
 and consider the ring  $R_1:=\C(x,q)[X]$
 of polynomials in $X$ with coefficients in
 the field $\C(x,q)$ of rational functions.  We define two elements of this ring by
$$
\mathcal{J}_2^1(X)=
(1-X_1X_2^7)\left(
\begin{aligned}
(1-x^2q^{12})(1-xq^6)
&-(1-xq^5)(1-x^2q^{13})X_3X_4^7
\\&+(1-q^{-1})(xq^6)(1-xq^7)
X_3^2X_4^{13}
\end{aligned}
\right)
$$
$$
\mathcal{J}_2^2(X)=
(1-X_1X_2^7)\left(
(1-x^2q^{12})(1-xq^6)(1-X_5^2X_6^{13})
-
X_3X_4^7(1-xq^5)(1-x^2q^{13})(1-X_5X_6^6)
\right),
$$
then we have
$$
J_2(B,C,E)
=
\frac{(1-xq^6)}{(1-xq^7)^2(1-x^2q^{13})}
 \begin{cases}
\mathcal{J}_2^1(x^{B+1},q^{B+1},x^{C+1},q^{C+1}, x^{E+1},q^{E+1}), \; E \ge C, \\
\mathcal{J}_2^2(x^{B+1},q^{B+1},x^{C+1},q^{C+1}, x^{E+1},q^{E+1}), \; E <C. \\
\end{cases}
$$

\label{section with J polynomials and T operators}
Let $R_2$ be the ring of Laurent polynomials
$\C(x,q)[X_1,X_2,X_3,X_4,X_1^{-1},X_2^{-1}].$
Then $\mathcal{J}_2^1 \in R_2,$ and
  each of the four summations above
corresponds to a linear operator
$R_1\to R_2.$
For example,
$$\begin{aligned}&
\sum_{\ell=1}^B
(x^{n_1}q^{n_2})^{B+1-\ell}
(x^{n_3}q^{n_4})^{C+1-\ell} (xq^8)^{\ell}\\&\qquad \qquad
=
(x^{n_1}q^{n_2})^{B+1}
(x^{n_3}q^{n_4})^{C+1}
\frac{x^{1-n_1-n_3}q^{8-n_2-n_4}-(x^{1-n_1-n_3}q^{8-n_2-n_4})^{B+1}}{1-x^{1-n_1-n_3}q^{8-n_2-n_4}},
\end{aligned}
$$for all $n_1, \dots n_4 \in \Z$ with
$(n_1+n_3, n_2+n_4) \ne (1,8).$
It follows that
$$
\sum_{\ell=1}^B
 J_2(B-\ell, C-\ell, \infty)
 = [T_1.\mathcal{J}_{2}^{1}](x^{B+1},q^{B+1}, x^{C+1}, q^{C+1}),
$$
where $T_1$ is the $\C(x,q)$-linear map
$R_1 \to R_2$ defined on monomials by
$$
T_1\left(
\prod_{i=1}^6 X_i^{n_i}\right)
= \prod_{i=1}^4 X_i^{n_i} (x^{1-n_1-n_3}q^{8-n_2-n_4}
- X_1^{1-n_1-n_3}Q_1^{8-n_2-n_4}).
$$
In similar fashion, we can define operators
corresponding to the other three summations
in \eqref{summations}.
Specifically, if
$$
T_2\left( \prod_{i=1}^6 X_i^{n_i}\right)
= \prod_{i=1}^4 X_i^{n_i} \frac{x^{2-n_1}q^{13-n_2}-X_1^{2-n_1}X_2^{13-n_2}}{1-x^{2-n_1}q^{13-n_2}},$$
$$\begin{aligned}
&T_3\left(\prod_{i=1}^6 X_i^{n_i}\right)
= \prod_{i=1}^4 X_i^{n_i}\cdot X_3^{n_5} X_4^{n_6}
\left[
\frac{x^{2-n_1-n_5}q^{14-n_2-n_6}}{(1-x^{2-n_1-n_5}q^{14-n_2-n_6})(1-x^{2-n_1}q^{13-n_2})}\right.\\&\left.
+\frac 1{1-x^{n_5}q^{n_6-1}}
\times \left( \frac{
-X_1^{2-n_1-n_5}X_2^{14-n_2-n_6}}{1-x^{2-n_1-n_5}q^{14-n_2-n_6}}
+\frac{
X_1^{2-n_1}X_2^{13-n_2}}{1-x^{2-n_1}q^{13-n_2}}
\right)\right]\end{aligned}
$$
$$
\begin{aligned}
&T_4\left(\prod_{i=1}^6 X_i^{n_i}\right)
= \prod_{i=1}^4 X_i^{n_i}\cdot X_3^{n_5} X_4^{n_6}
\left(
\frac{x^{3-2n_1-n_3-2n_5}q^{22-2n_2-n_4-2n_6}}
{(1-x^{1-n_1-n_3-n_5}q^{8-n_2-n_4-n_6})
(1-x^{2-n_1-n_5}q^{14-n_2-n_6})}\right.\\
&\left.
-\frac{x^{1+n_3}q^{6+n_4}X_1^{1-n_1-n_3-n_5}X_2^{8-n_2-n_4-n_6}}{(1-x^{1-n_1-n_3-n_5}q^{8-n_2-n_4-n_6})(1-x^{1+n_3}q^{6+n_4})}
+\frac{X_1^{2-n_1-n_5}X_2^{14-n_2-n_6}}
{(1-x^{2-n_1-n_5}q^{14-n_2-n_6})(1-x^{1+n_3}q^{6+n_4})}
\right).\\
\end{aligned}
$$
Then $J_0(p^B,p^C) = \mathcal{J}_0(x^{B+1}, q^{B+1}, x^{C+1},q^{C+1}),$ where
$$
\mathcal{J}_0=\frac{1-xq^6}{(1-xq^7)^2(1-x^2q^{13})}[\mathcal{J}_2^1+(1-q^{-1})(T_1+T_2).\mathcal{J}_2^1+(1-q^{-1})^2(T_3+T_4).\mathcal{J}_2^1].
$$
\begin{prop}  The value of
$\c J_0$
is given by
$$\begin{aligned}
\frac{(1-xq^6)(1-x^2q^{12})}{(1+xq^7)^3(1-x^2q^{13})}\left(
\frac{(1-xq^6)(1-x^2q^{13})}{
   \left(1-q^8 x\right)}-\frac{X_1X_2^8
   \left(1-xq^5\right)\left(1-x^2q^{14}\right)}{
   \left(1-xq^8 \right)}+
(1-q^{-1}) xq^6  X_1^2 X_2^{14}
   \right)\\
   +X_3X_4^7(1 - q^5 x) \left(
   -(1 + q^6 x) (1 - q^7 x)X_2q^{-1}
   +  (1 - q^{13} x^2) q^{-1} X_1X_2^7
   -(1 - q^{-1}) q^6 x X_1^2 X_2^{14}
   \right)\\
   +X_3^2X_4^{13} \frac{(1 - q^{-1})
}{x q^7 (1 - q^8 x)} \left((1 - q^8 x) (1 +  x^2q^{12} -  x^2q^{13} -  x^3q^{19})X_2\right.\\\left. \qquad+ x q^8(1 -  x q^4-  x q^6+  x^2 q^{11}+  x^2 q^{12}-  x^2q^{13})X_1X_2^8\right)
   \end{aligned}
$$
\end{prop}
\begin{proof}
Straightforward calculations give each of the
components of the sum.  We record the
results:
$$\begin{aligned}{}
[T_1.\c J_2^1]=&(1-xq^6)(1-x^2q^{12})\left(
\frac{xq^8}{1-xq^8}-\frac q{1-q} X_1X_2^7 + \frac{q(1-xq^7)}{(1-xq^8)(1-q)} X_1X_2^8 \right)\\ &
- (1-xq^5)(1-x^2q^{13})
\left( \frac q{1-q} + \frac 1{1-xq^6} X_1X_2^7
- \frac{1-xq^7}{(1-q)(1-xq^6)} X_2
\right)X_3 X_4^7\\&\hskip -30pt
+(1-q^{-1})(1-xq^7)(xq^6) \left(
\frac{-1}{1-xq^5}+\frac 1{1-x^2 q^{12}} X_1X_2^7 + \frac{xq^5(1-xq^7)}{(1-xq^5)(1-x^2q^{12})} X_1^{-1} X_2^{-5}
\right)X_3^2X_4^{13}\end{aligned}
$$
$$\begin{aligned}
 {}[T_2.\c J_2^1] &= [T_2.(1-X_1X_2^7)] \c J_4^1=
 \left(\frac{ x^2q^{13}}{1- x^2q^{13}}-\frac{ x q^6X_1 X_2^7}{1-
   xq^6}+\frac{ xq^6 X_1^2 X_2^{13} \left(1-
   xq^7\right)}{\left(1- xq^6\right) \left(1- x^2q^{13}\right)}
 \right)\c J_4^1
\end{aligned}$$
$$\begin{aligned}{}
[T_3. \c J_2^2]=&
\left((1-xq^6)(1-x^2q^{12}) -(1-xq^5)(1-x^2q^{13})X_3X_4^7\right)\cdot [T_3. (1-X_1X_2^7)]
+ c_2(X_1, X_2)X_3^2 X_4^{13},\end{aligned}
$$ where
$$\begin{aligned}{}
T_3. (1-X_1X_2^7)&=
\frac{x^2q^{14}}{(1-x^2q^{14})(1-x^2q^{13})}- \frac{X_1X_2^7xq^7}{(1-xq^7)(1-xq^6)}\\&
-\frac{X_1^2X_2^{13} xq^6(1-xq^7)}{(1-q^{-1})(1-xq^6)(1-x^2q^{13}) }
+\frac{X_1^2X_2^{14}xq^7}{(1-q^{-1})(1-x^2q^{14})}
\end{aligned}$$
$$\begin{aligned}
c_2= (1-xq^5)(1-x^2q^{13})&\left[
\frac{xq^8}{(1-xq^8)(1-x^2q^{13})}-\frac{X_1X_2^7q}{(1-xq^6) (1-q)}\right.\\ {}&\left.
+\frac{X_1X_2^8q}{(1-xq^5)(1-xq^8)}
-\frac{X_1X_2^{13}xq^6(1-xq^7)}{(1-xq^6)(1-x^2q^{13})(1-xq^5)}
\right]\\
-(1-xq^6)(1-x^2q^{12}])&\left[
\frac{q}{(1-q)(1-x^2q^{13})}-\frac{X_1X_2^7x^{-1}q^{-6}}{(1-xq^6)(1-x^{-1}q^{-6})}\right.\\&\left.
+\frac{X_2 x^{-1}q^{-6}}{(1-x^2q^{12})(1-q)} -\frac{X_1^2X_2^{13}xq^6(1-xq^7)}{(1-xq^6)(1-x^2q^{13})(1-x^2q^{12})}
\right]
\end{aligned}$$
Finally
$$\begin{aligned}
{}[T_4.\c J_2^2]=&
(1-xq^6)(1-x^2q^{12})Q(0,0,0,0)
-(1-xq^5)(1-x^2q^{13})Q(1,7,0,0)X_3X_4^7\\&
-[(1-xq^6)(1-x^2q^{12})Q(2,13,0,0)-(1-xq^5)(1-x^2q^{13})Q(1,7,1,6)]X_3^2X_4^{13}.
\end{aligned}
$$
where
$$\begin{aligned}
Q(n_3,n_4,n_5,n_6)&=
 \frac {x^{3-{n_3}-2 {n_5}}
   q^{22-{n_4}-2 {n_6}}}{\left(1-x^{2-{n_5}}
   q^{14-{n_6}}\right) \left(1-x^{1-{n_3}-{n_5}}
   q^{8-{n_4}-{n_6}}\right)}\\&\qquad
-\frac {X_1 X_2^7 x^{1-n_3-2 n_5}
   q^{8-n_4-2n_6}}{\left(1-x^{1-{n_5}}
   q^{7-{n_6}}\right) \left(1-x^{-{n_3}-{n_5}}
   q^{1-{n_4}-{n_6}}\right)}\\&\qquad\quad- \frac {x^{1-{n_5}}q^{7-{n_6}}
   X_1^{2-{n_5}}  X_2^{14-{n_6}}
   \left(1-q^7 x\right)}{\left(1-x^{{n_3}+1} q^{{n_4}+6}\right)
   \left(1-x^{1-{n_5}} q^{7-{n_6}}\right)
   \left(1-x^{2-{n_5}}
   q^{14-{n_6}}\right)}\\&\qquad\qquad+ \frac {x^{1-{n_5}} q^{7-{n_6}}
   X_1^{1-{n_3}-{n_5}}
   X_2^{8-{n_4}-{n_6}} \left(1-q^7 x\right)}{\left(1-x^{{n_3}+1}
   q^{{n_4}+6}\right) \left(1-x^{-{n_3}-{n_5}}
   q^{1-{n_4}-{n_6}}\right) \left(1-x^{1-{n_3}-{n_5}}
   q^{8-{n_4}-{n_6}}\right)}.\end{aligned}
$$
Totalling up the contributions, multiplying by
$
\frac{(1-xq^6)}{(1-xq^7)^2(1-x^2q^{13})},
$ and simplifying gives the result.  \end{proof}
\subsection{Final calculation of $I(s,t)$}
In this section we complete
the proof of theorem \ref{thm:  final formula for I(s,t)}.
First, by proposition \ref{prop J in terms of J'},
$J(a,b,c) =
P_0(x,q)
(1-
(xq^7)^{A+1}
)
J_0(b,c),$ where
$$
P_0(x,q)=
\frac
{(1-x)(1-xq^2)(1-xq^3)(1-xq^4)(1-x^2q^{10})}{(1-xq^6)(1-xq^7)},$$
and by theorem \ref{theorem: I(s,t) in terms of J(a,b,c)}
$$I(s,t)
=\begin{cases}
J(1,1,1)(1+x^3q^{18}),&t_1, t_2 \in \f o ^\times\\
J(1,1,t_1)-x^4q^{26}J(1,1,\frac{t_1}{p^2}),&
t_2 \in \f o^\times, t_1 \notin \f o ^\times\\
J(1,t_2,t_1t_2)-x^2q^{14}J(p,\frac{t_2}p,\frac{t_1t_2}p)+x^5q^{35} J(1,\frac{t_2}{p^2},\frac{t_1t_2}{p^2}),
& t_2 \notin \f o^\times.
\end{cases}
$$

\subsubsection{Case 1:  $t_2 \notin \f o ^\times$}
First, assume  $t_2 \notin \f o ^\times,$ define
$B$ and $C$ in terms of $t_1, t_2$ by
$|t_2|=q^{-B}, \; |t_1 t_2|=q^{-C}.$  Then
$$\begin{aligned}
I(s,t)
&=J(1,t_2,t_1t_2)-x^2q^{14}J(p,\frac{t_2}p,\frac{t_1t_2}p)+x^5q^{35} J(1,\frac{t_2}{p^2},\frac{t_1t_2}{p^2})\\
&=(1-xq^7)\left(J_0(t_2,t_1t_2)-x^2q^{14}(1+xq^7)J_0(\frac{t_2}p,\frac{t_1t_2}p)+x^5q^{35} J_0(\frac{t_2}{p^2},\frac{t_1t_2}{p^2})\right)\\
&= (1-xq^7)[T_0. \mathcal J_0](x^{m+1}, q^{m+1}, x^{n+m+1},q^{n+m+1}),
\end{aligned}
$$
where
$n$ and $m$ are the $\f p$-adic valuations
of $t_1$ and $t_2,$ respectively, and $T_0$
is a linear operator on the ring $R_2$ defined by
$$
T_0 \left( \prod_{i=1}^4 X_i^{n_i}\right)
= \prod_{i=1}^4 X_i^{n_i}(1-x^{2-n_1-n_3}q^{14-n_2-n_4})(1-x^{3-n_1-n_3}q^{21-n_2-n_4}).
$$
Applying the operator $T_0$
to $\c J_0$ as computed in the previous
section gives
$$
{}[T_0.\c J_0]=\frac{(1-xq^6)(1-x^2q^{12})}{(1-xq^7)(1-xq^8)}
   \left[\begin{aligned}
   (1-xq^6)(1-x^3q^{21})
   - (1-xq^6)(1-xq^8)X_3X_4^7\hskip .7in \\
   - q^{-1} X_2X_3X_4^7 (1-xq^5)(1-xq^8)
   \end{aligned}
   \right].
$$
Multiplying by $(1-xq^7)P_0(x,q)=(1-xq^6)(1-x^2q^{12})Z(x,q),$ and plugging in $X_2=q^{m+1}, X_3=x^{n+m+1}$
and $X_4=q^{n+m+1}$
 gives
 the proof in this case.

\subsubsection{Case 2: $t_2 \in \f o^\times$}
In this case the summations
corresponding to $T_1, T_3$ and $T_4$
are empty sums.  Hence these terms
will specialize to zero in that case.
Moreover
$$(\c J_{2}^{1}+ (1-q^{-1}) T_2. \c J_2^1)
=\frac{1}{(1-xq^6)(1-x^2q^{13})} \c J_4^1(x,q,X_1,X_2)
\c J_4^1(x,q,X_3,X_4).$$
The condition $t_2 \in \f o^\times$ translates
to  $X_1=x, X_2=q,$
and
$$
\c J_4^1(x,q,X_1,X_2) =(1-xq^6)(1-xq^7)(1-x^2q^{13}),
$$
so we obtain
$\c J_0= \frac{1-xq^6}{(1-xq^7)(1-x^2q^{13})}\c J_4^1(x,q,X_3,X_4)$
in this case.
Now,
$$
I=
P_0(x,q)(1-xq^7)
\begin{cases}(1+x^3q^{18})
\c J_0 (x,q,x,q), & t_1 \in \f o^\times, \\
\c J_0 (x,q,X_3,X_4) - x^4 q^{26} \c J_0 ( x,q, X_3 x^{-2}, X_4 q^{-2}), & t_2 \notin \f o^\times,
\end{cases}
$$
Simplifying
$$
\c J_0 (x,q,X_3,X_4) - x^4 q^{26} \c J_0 ( x,q, X_3 x^{-2}, X_4 q^{-2})$$
$$= \frac{1-xq^6}{(1-xq^7)(1-x^2q^{13})}
\left( (1-xq^6)(1-x^2q^{12})(1-x^4q^{26})
- (1-xq^5)(1-x^2q^{13})(1-x^2q^{12})X_3X_4^7\right).
$$
If $X_3=x,$ and $x=q,$ this
coincides with $(1+x^3q^{18}) \c J_0(x,q,x,q).$
It follows that the case $t_1 \in \f o^\times$ does
not need to be handled separately.
After incorporating $(1-xq^7)P_0(x,q),$
one has only to check that
$$I_0(n,0; x,q) = (1-xq^8)\left(
(1-xq^6)(1+x^2q^{13})-(1-xq^5)x^{n+1}q^{7n+7}
\right),$$
and this is straightforward.

\end{document}